\documentclass[12pt, twoside, a4paper]{amsart}

\usepackage{amssymb}
\usepackage{geometry}
\usepackage[T1]{fontenc}
\usepackage[ngerman,english]{babel}
\usepackage[latin1]{inputenc}

\usepackage[leqno]{amsmath}
\usepackage{amsthm}
\usepackage{amsfonts}

\newtheorem{Def}{Definition}[section] 

\newtheorem{thm}[Def]{Theorem}
\newtheorem{prop}[Def]{Proposition}
\newtheorem{lem}[Def]{Lemma}
\newtheorem{kor}[Def]{Corollary}
\newtheorem{satz}[Def]{Proposition}
\newtheorem*{Theo}{Theorem}
\newtheorem*{Prop}{Proposition}
\newtheorem*{Cor}{Corollary}

\theoremstyle{definition} 

\newtheorem{deff}[Def]{Definition}
\newtheorem{Not}[Def]{Notation}

\newtheorem{bem}[Def]{Remark}

\usepackage[numbers,square]{natbib}

\begin{document}

\title[Spectrum of generalized Hodge-Laplace operators on tori and spheres]{Spectrum of generalized Hodge-Laplace operators on flat tori and round spheres}

\author{Stine Franziska Beitz}



\address{WWU Münster, Mathematisches Institut, Orléansring 12, 48149 Münster,
Germany} 
\email{beitzs@uni-muenster.de}

\begin{abstract}
We consider generalized Hodge-Laplace operators $\alpha d \delta + \beta \delta d$ for $\alpha, \beta > 0$ on $p$-forms on compact Riemannian manifolds. In the case of flat tori and round spheres of different radii, we explicitly calculate the spectrum of these operators. Furthermore, we investigate under which circumstances they are isospectral.
\end{abstract}

\maketitle



\section{Introduction}

Can one hear the shape of a drum? Can one draw conclusions on its shape just based on its sound?
    This question was raised by Mark Kac already in 1966 (\cite{kac}). However, he was not able to answer it completely. Mathematically, in his paper, a clamped elastic membrane is modeled by a domain $G$ in the plane. The resonance frequencies are just the eigenvalues of the Dirichlet problem of the Laplacian $\Delta_0$ on functions, i.e.\ those real numbers $\lambda$, for which there are functions $f: \overline{G} \rightarrow \mathbb{R}$, $f \not\equiv 0$ which vanish on the boundary of $G$ and comply with the eigenvalue equation $\Delta_0 f = \lambda f$.\\
In spectral geometry, similar problems are investigated in a more general setting. Here one is interested in the relationship between the geometric structures of Riemannian manifolds and the spectra of elliptic differential operators. In particular, one wonders which information the spectrum of these operators provides about the geometry of the underlying manifolds. One of the first results of this type was discovered by Hermann Weyl 1911 (\cite{weyl}). He showed that the volume of a bounded domain in the Euclidian space is determined by the asymptotic behaviour of the eigenvalues of the Dirichlet problem of the Laplace-Beltrami operator. To reconstruct a manifold completely up to isometry, the knowledge of the eigenvalues is, however, not sufficient. The answer to the question ``Can One Hear the Shape of a Drum?''\ therefore is ``no''. This was already recognized by John Milnor, who proved the existence of two non-isometric 16-dimensional tori with identical spectra of the Laplacians (\cite{jm}). Later, Gordon, Webb and Wolpert constructed different domains in the plane for which the eigenvalues coincide (\cite{gww}).   

\medskip
In this paper, we consider the family of differential operators $F_{\alpha \beta}^M := \alpha d \delta + \beta \delta d$ for real numbers $\alpha, \beta > 0$ on differential forms on a compact Riemannian manifold $(M, g)$ of dimension $n$. Here $d$ denotes the exterior derivative on differential forms and $\delta$ the adjoint operator of $d$. For $\alpha=\beta=1$ this yields the well-known Hodge-Laplace operator. Our aim is to determine the spectrum $\mathrm{Spec}(F_{\alpha \beta}^M)$ of the operators $F_{\alpha \beta}^M$ explicitly for certain manifolds and to investigate under which circumstances these operators are isospectral, i.e.\ possess the same spectra. 

\medskip
The investigation of the spectrum of the operator $F_{\alpha \beta}^M$ finds applications, for example, in elasticity theory, where the operator appears in the classical problem of linear electrodynamics (\cite[Section 61.10, p.\ 212]{ez}). Physically, this is relevant when computing the oscillation frequencies of an elastic body (for bodies made of simple materials); it turns out that these are determined by the eigenvalues of the operator in linear approximation.

\medskip
The operator $F_{\alpha \beta}^M$ is elliptic and self-adjoint on its domain of definition. Together with the compactness of $M$ it follows (see for example \cite[Lemma 1.6.3]{gilkey}) that its spectrum is discrete and only consists of real eigenvalues of finite multiplicities. Furthermore, the associated eigenforms are smooth. Hence, to determine the spectrum of the operator $F_{\alpha \beta}^M$ it suffices to investigate the algebraic eigenvalue problem, i.e.\ to find solutions $\lambda \in \mathbb{R}$ and $\omega \in \Omega^p(M)$ of the eigenvalue equation $F_{\alpha \beta}^M\omega = \lambda \omega$. Moreover, we just need to understand $F_{\alpha \beta}^M$ as an operator on the smooth $p$-forms $\Omega^p(M)$.

\medskip
In this work, we will calculate the spectrum, that is to say, the eigenvalues with the associated eigenspaces, of the operators $F_{\alpha \beta}^M$ for two sample-manifolds $M$ $-$ the spheres $S^n_r$ of different radii $r>0$ and the flat tori $\mathbb{R}^n/\Lambda$ induced by lattices $\Lambda \subset \mathbb{R}^n$ (see theorems \ref{thmspt} and \ref{thms}). We will thereby discover that the spectrum splits into eigenvalues of the operators $\alpha d \delta$ and $\beta \delta d$, which depend linearly on $\alpha$ and $\beta$, respectively. 

\begin{Prop}
If $n=2p$, the spectrum is symmetric in $\alpha$ and $\beta$. 
\end{Prop}

It will turn out that the eigenvalues of $F_{\alpha \beta}^{T}$ on flat tori $T$ are just $\alpha \lambda$ and $\beta \lambda$ where $\lambda$ are the eigenvalues of the ordinary Laplacian $\Delta_0$ on smooth functions. With this knowledge statements about the isospectrality of two operators $F_{\alpha \beta}^{T}$ and $F_{\alpha' \beta'}^{T'}$ for $\alpha, \alpha', \beta, \beta' > 0$ on flat tori $T$ and $T'$ can be made. \\
First we will fix the coefficients $\alpha$ and $\beta$ and notice that obviously for two compact isometric Riemannian manifolds $M$ and $N$ the operators $F_{\alpha \beta}^M$ and $F_{\alpha \beta}^N$ have the same spectrum. Then the question arises whether the converse of this statement in the case of flat tori holds as well, i.e.\ whether the spectrum of $F_{\alpha \beta}^{T}$ already determines the flat torus $T$ up to isometry. In dimension $n=1$, that is to say for 1-dimensional tori, one directly sees that the answer is ``yes''\ \!\!. Next we prove the following theorem for two flat tori $T_1$ and $T_2$. We remark that in the case $\alpha = \beta$, this was already observed by Milnor in \cite{jm}. 

\begin{Theo}
The operators $F_{\alpha \beta}^{T_1}$ and $F_{\alpha \beta}^{T_2}$ are isospectral if and only if the Laplacians $\Delta_0^{T_1}$ and $\Delta_0^{T_2}$ have the same spectrum. 
\end{Theo}
 
 Thereby we also can answer the above question for higher dimensions since the answer in the case of the Laplacian on functions is already known: In dimensions $n=2, 3$ it is ``yes''\ as well (see the proof of Proposition \ref{PropOpIsospToriIsom} below). For higher dimensions, we have the following result:
 
\begin{Prop}
From dimension $n=4$ onward there are non-isometric flat tori the spectra of which with respect to $F_{\alpha \beta}$  coincide. 
\end{Prop} 
 
On the other hand, fixing a flat torus $T$ of dimension $n \neq 2p$ and allowing arbitrary parameters $\alpha, \alpha', \beta, \beta' > 0$, we find that 

\begin{Theo}
$F_{\alpha \beta}^{T}$ and $F_{\alpha' \beta'}^{T}$ can only be isospectral in the trivial case that these operators are already the same. 
\end{Theo}

Finally, we consider two flat tori, the lattices of which are related by stretching by a factor $c>0$. We show: 

\begin{Cor}
$F_{\alpha \beta}^{\mathbb{R}^n/\Lambda}$ and $F_{\alpha' \beta'}^{\mathbb{R}^n/c\Lambda}$ have the same spectrum if and only if $(\alpha', \beta') = (c^2\alpha, c^2 \beta)$. 
\end{Cor}

In dimension $n=2p$, similar statements hold which take into consideration that the spectra of the operators $F_{\alpha \beta}$ are symmetric in $\alpha$ and $\beta$, as already mentioned above. \\
The question of what can be said about the spectra of $F_{\alpha \beta}^{\mathbb{R}^n/\Lambda}$ and $F_{\alpha' \beta'}^{\mathbb{R}^n/\Lambda'}$ for arbitrary lattices $\Lambda$ and $\Lambda'$ remains open. For $n=2$, it is tempting to conjecture that these coincide if and only if there exists a $c> 0$ and a $Q \in O(n)$ such that $\Lambda' = cQ\Lambda$ and $\{\alpha', \beta'\}=\{c^2\alpha, c^2\beta\}$.  
 
\medskip
In contrast to flat tori where the eigenvalue 0 always has the multiplicity $\binom{n}{p}$, the eigenvalues on spheres are all positive for $p\neq 0$. Here again, we compare the operators $F_{\alpha \beta}^{S^n_r}$ and $F_{\alpha' \beta'}^{S^n_{r'}}$. 

\begin{Prop}
If these operators have the same spectra, then the radii $r$ and $r'$ are equal if and only if $\alpha=\alpha'$ and $\beta=\beta'$ (for $n=2p$, up to exchange of the roles of $\alpha$ and $\beta$). 
\end{Prop}

For $c>0$ such that $r'=cr$ we can show: 

\begin{Prop}
The isospectrality of the two above mentioned operators is equivalent to $(\alpha', \beta')=(c^2\alpha, c^2\beta)$ (for $n=2p$, again up to exchange of the roles of $\alpha$ and $\beta$).
\end{Prop}

\subsubsection*{Acknowledgements} 
I would like to express my special thanks to Prof.\ Dr.\ Christian Bär for proposing the topic of this paper, for lots of stimulating conversations and many useful hints and references. It is my pleasure to thank Dr. Andreas Hermann for his advice and many valuable suggestions. Moreover, I am grateful to Matthias Ludewig for his support and for numerous helpful discussions.


\section{Preliminaries}

Let $M$ be an $n$-dimensional compact Riemannian manifold, $\alpha, \beta > 0$ and $0 \leq p \leq n$. 

\medskip
Throughout the paper, we write $\Omega^p(M) := \Gamma(M,\Lambda^pT^*M)$ for the space of all smooth differential forms of degree $p$ on $M$ and, for its complexification, we write $\Omega^p(M, \mathbb{C}) := \Omega^p(M) \otimes \mathbb{C}$.
We denote by $\Omega^p_{L^2}(M):= L^2(M, \Lambda^pT^*M)$ the completion of the smooth $p$-forms with respect to the $L^2$-scalar product $(\cdot, \cdot)$, and we write $\Omega^p_{L^2}(M, \mathbb{C}) := \Omega^p_{L^2}(M) \otimes \mathbb{C}$ for its complexification.

\medskip
Now we introduce the central object of this paper.

\begin{deff}
We define the operator
\begin{align*}
F_{\alpha \beta, p}^M := \alpha d\delta + \beta \delta d
\end{align*}
on the space $\Omega^p(M)$.  Here, $d$ is the exterior derivative on differential forms and $\delta$ the formal adjoint of $d$, i.e.\ for all $\omega \in \Omega^p_c(M)$ and $\eta \in \Omega^{p+1}_c(M)$ we have $(d\omega, \eta) = (\omega, \delta \eta)$. We simply write $F_{\alpha \beta}^M$ if it is clear from the context on which forms the operator is considered.\\
The Hodge-Laplace operator on $\Omega^p(M)$ we denote by $\Delta^M_p:= F_{11,p}^M = d\delta + \delta d$.
\end{deff}

\begin{bem} \label{syor} 
$ $
\begin{enumerate}
\item The principal symbol of $F_{\alpha \beta}^M$ is
\begin{align*}
- \beta |\xi|^2 \mathrm{id} - (\alpha - \beta) \xi \wedge (\xi^{\sharp} \lrcorner \ \cdot) ,
\end{align*}
where $\xi \in T^*M$. Here $\xi^{\sharp} \lrcorner \omega := \omega(\xi^{\sharp}, \cdot, ..., \cdot)$ for $\omega \in \Omega^p(M)$ and $\sharp$ denotes the musical isomorphism which assigns to each $\xi \in T^*M$ the uniquely determined vector $\xi \in TM$ such that $g(\xi^{\sharp}, \cdot) = \xi$. For each $\xi \neq 0$, the principal symbol is invertible with inverse map
\begin{align*}
- \frac{1}{\beta|\xi|^2} \mathrm{id} + \frac{\alpha- \beta}{\alpha \beta |\xi|^4} \xi \wedge (\xi^{\sharp} \lrcorner \ \cdot) .
\end{align*}
This shows that $F_{\alpha \beta}^M$ is elliptic.
\item It is easy to check that $F_{\alpha \beta}^M$ is formally self-adjoint. It then follows from the general theory of elliptic operators on compact manifolds that $F_{\alpha \beta}^M$ considered as an unbounded operator on $\Omega_{L^2}^p(M)$ with domain $\Omega_{L^2}^p(M)$ has a unique self-adjoint extension, with domain the Sobolev space $H^2(M, \Lambda^pT^*M)$. As such, it has a discrete spectrum, finite-dimensional eigenspaces, and its eigenvalues tend to infinity.
\end{enumerate}
\end{bem}

\begin{bem}
The exterior differential $d$ and the co-derivative $\delta$ can be written in terms of the connection on forms in the following way: Let $\{e_1, ..., e_n\}$ be a local orthonormal basis of $TM$ and $\{e^1, ..., e^n\}$ the associated dual basis of $T^*M$. Then 
\begin{align} \label{Formulasddelta}
d=\sum_{i=1}^n e^i \wedge \nabla_{e_i} \ \ \ \text{and} \ \ \
\delta = - \sum_{i=1}^n e_i \lrcorner \nabla_{e_i} .
\end{align}
\end{bem}

In order to understand the spectrum of $F_{\alpha \beta}^M$ (in the case that $M$ is compact) as a set of eigenvalues with associated multiplicities, we introduce the concept of a weighted set, which is essentially a set (here we specialize to subsets of $\mathbb{C}$) in which elements can be contained more than once; this can be modelled by a function on $\mathbb{C}$ with values in $\mathbb{N}_0$; to $\lambda \in \mathbb{C}$, the function assigns the number of times that $\lambda$ is contained the set described by it.

\begin{deff} $ $
\begin{enumerate} 
\item A weighted set is a function $W : \mathbb{C} \rightarrow \mathbb{N}_0$. 
\item If $W$ has a countable support $\mathrm{supp}(W) := \{\lambda \in \mathbb{C} \mid W(\lambda) \neq 0\} = \{ \lambda_i \mid i \in \mathbb{N}\}$, we write 
\begin{align*}
W := \{ (\lambda_1, W(\lambda_1)), (\lambda_2, W(\lambda_2)), ... \} 
\end{align*}
and respectively
\begin{align*}
W := \{ \underbrace{\lambda_1, ..., \lambda_1}_{W(\lambda_1)\text{-times}}, \underbrace{\lambda_2, ...,\lambda_2}_{W(\lambda_2)\text{-times}}, ...\} .
\end{align*} 
\item Let $W$ and $W'$ be weighted sets. Then their weighted union $W \Cup W'$ is defined as:
\begin{align*}
(W \Cup W')(\lambda) := W(\lambda) + W'(\lambda)
\end{align*}
for all $\lambda \in \mathbb{C}$.
\item In addition, we introduce the following notation for $m, m' \in \mathbb{N}_0$:
\begin{align*}
W \ ^m \! \Cup^{m'} W' := \underbrace{W \Cup ... \Cup W}_{m \ \text{-times}} \Cup \underbrace{W' \Cup ... \Cup W'}_{m' \ \text{-times}} .
\end{align*}
For $m=1$ or $m'=1$ we usually omit the index. 
\item The difference of $W$ and $W'$ is the weighted set $W \setminus W'$, which is defined by 
\begin{align*}
(W \setminus W')(\lambda) := \max \{W(\lambda) -W'(\lambda), 0 \}
\end{align*}
for all $\lambda \in \mathbb{C}$.
\item The minimum of a weighted set $W$ with support $\mathrm{supp}(W) \subseteq \mathbb{R}$ is given by $\min (W) := \min \! \big(\mathrm{supp}(W)\big)$ .
\item For $r \in \mathbb{R}^*$ let $r W(\lambda) := W(\frac{\lambda}{r})$ for all $\lambda \in \mathbb{C}$.
\end{enumerate}
\end{deff}

\begin{deff} Let $M$ be compact.
\begin{enumerate}
\item We call
\begin{align*}
\mathrm{Eig}(F_{\alpha \beta,p}^M, \lambda) := \{ \omega \in \Omega^p(M) \mid F_{\alpha \beta,p}^M\omega = \lambda \omega \}
\end{align*}
the eigenspace of $F_{\alpha \beta,p}^M$ to the eigenvalue $\lambda$.
\item The spectrum of $F_{\alpha \beta,p}^M$ is the weighted set for $\lambda \in \mathbb{C}$ defined by
\begin{align*}
\mathrm{Spec}(F_{\alpha \beta,p}^M)(\lambda) := \dim \! \big(\mathrm{Eig}(F_{\alpha \beta,p}^M, \lambda)\big) .
\end{align*}
\end{enumerate}
\end{deff}

\begin{bem}
The measure 
\begin{align*}
\sum_{\lambda \in \mathbb{C}} \mathrm{Spec}(F_{\alpha \beta,p}^M)(\lambda) \delta_{\lambda},
\end{align*}
where $\delta_{\lambda}$ is the Dirac measure at $\lambda \in \mathbb{C}$, is exactly the spectral measure of $F_{\alpha \beta,p}^M$.
\end{bem}

\begin{prop} \label{symn=2}
Let $M$ be an $n$-dimensional orientable Riemannian manifold, $\alpha, \beta > 0$ and $0 \leq p \leq n$. Then the spectra of $F_{\alpha \beta, p}^M$ and $F_{\beta \alpha, n-p}^M$ coincide.
\end{prop}

\begin{proof}
Let $*_p : \Omega^p(M) \rightarrow \Omega^{n-p}(M)$ be the Hodge-Star operator on $p$-forms. Using the properties
\begin{align*}
*_{n-p}*_p = (-1)^{p(n-p)}  \hspace{1em} \text{and} \hspace{1em} \delta = (-1)^{p+1} *_{p+1}d*_p^{-1},
\end{align*}
one can show that
\begin{align*}
*_p F_{\alpha \beta, p}^M *_p^{-1} = F_{\beta \alpha, n-p}^M .
\end{align*}
Using this, it is then easy to see that $*$ gives a bijection between the eigenspaces $\mathrm{Eig}(F_{\alpha \beta, p}^M, \lambda)$ and $\mathrm{Eig}(F_{\beta \alpha, n-p}^M, \lambda)$.\end{proof}


\section{Spectrum on flat tori}

In this section, let $\alpha, \beta > 0$ and $1 \leq p \leq n$. We will investigate and determine explicitely the spectrum of  $F_{\alpha \beta,p}^{T^n}$ on $n$-dimensional flat tori $T^n$.

\begin{deff}
Let $B:= \{b_1, ..., b_n\}$ be a basis of $\mathbb{R}^n$ and $\{b^1, ..., b^n\}$ the associated dual basis of $({\mathbb{R}^n})^*$, i.e.\ $b^i(b_j) = \delta^i_j$ for all $i,j \in \{1, ..., n\}$. Then
\begin{align*}
\Lambda_B := \mathbb{Z}b_1 + ... + \mathbb{Z}b_n
\end{align*}
is the lattice induced by $B$ and
\begin{align*}
\Lambda_B^* := \mathbb{Z}b^1 + ... + \mathbb{Z}b^n = \{ l \in (\mathbb{R}^n)^* \mid \forall \lambda \in \Lambda : l(\lambda) \in \mathbb{Z} \} 
\end{align*}
the dual lattice of $\Lambda_B$.
\end{deff}

Let $\Lambda$ be a lattice in $\mathbb{R}^n$ and $\pi : \mathbb{R}^n \rightarrow \mathbb{R}^n/\Lambda : x \mapsto [x]$ the canonical projection on $\mathbb{R}^n/\Lambda$. Then $(\mathbb{R}^n/\Lambda, g)$ is an $n$-dimensional flat torus associated to the lattice $\Lambda$, which we often denote by $T^n$, where $\mathbb{R}^n/\Lambda$ is endowed with the Riemannian metric $g$ induced by the standard metric $g_{\mathrm{std}}$ on $\mathbb{R}^n$ via $g_{\mathrm{std}} = \pi^* g$. 
We remark that flat tori are Lie groups with respect to the usual addition. \\
We write $\{\partial_1, ..., \partial_n\}$ and $\{dx^1, ..., dx^n\}$ for the global bases of $TT^n$ and $T^*T^n$ induced by the standard basis of $\mathbb{R}^n$. \\
Due to the existence of a global basis of $T^*T^n$, for flat tori $T^n$ we have the following isomorphism:
\begin{align} \label{tensoriso} \begin{split}
L^2(T^n,\mathbb{C}) \otimes \Lambda^p(\mathbb{R}^n)^* &\longrightarrow \Omega^p_{L^2}(T^n, \mathbb{C}) \\
\sum_{i_1, ..., i_p=1}^n f_{i_1...i_p} e^{i_1} \wedge ... \wedge e^{i_p} &\longmapsto \sum_{i_1, ..., i_p=1}^n f_{i_1...i_p} dx^{i_1} \wedge ... \wedge dx^{i_p} , \end{split}
\end{align}
where $\{e^1, ..., e^n\}$ is the dual basis of the standard basis in $\mathbb{R}^n$.

\begin{deff}
Let $\{e^1, ..., e^n\}$ be the dual basis of the global standard basis of $T\mathbb{R}^n$ and $T^n$ a flat torus. Then 
\begin{align*}
\Omega^p_{\mathrm{par}}(T^n) &:= \{ \omega \in \Omega^p(T^n) \mid \nabla \omega = 0\} \\ 
&= \left\{ \left. \sum_{i_1, ..., i_p=1}^n a_{i_1...i_p} dx^{i_1} \wedge ... \wedge dx^{i_p} \ \right| \  a_{i_1...i_p} \in \mathbb{R} \right\}
\end{align*}
and
\begin{align*}
\Omega^p_{\mathrm{par}}(\mathbb{R}^n) := \left\{ \left. \sum_{i_1, ..., i_p=1}^n a_{i_1...i_p} e^{i_1} \wedge ... \wedge e^{i_p} \ \right| \  a_{i_1...i_p} \in \mathbb{R} \right\}
\end{align*}
are the spaces of all parallel $p$-forms on $T^n$ and $\mathbb{R}^n$ respectively. \\
Furthermore let $\Omega^p_{\mathrm{par}}(T^n, \mathbb{C}) := \Omega^p_{\mathrm{par}}(T^n) \otimes \mathbb{C}$ and $\Omega^p_{\mathrm{par}}(\mathbb{R}^n, \mathbb{C}) := \Omega^p_{\mathrm{par}}(\mathbb{R}^n) \otimes \mathbb{C}$.
\end{deff}

\begin{bem} \label{RemarkIsoFormsOnRn}
For flat tori $T^n$, we have
\begin{align*}
\Omega^p_{\mathrm{par}}(T^n) &\cong \Lambda^p(\mathbb{R}^n)^*  \\
\text{by virtue of} \ \ \ \ \ \ \sum_{i_1, ...i_p=1}^n a_{i_1...i_p} dx^{i_1} \wedge ... \wedge dx^{i_p} &\longmapsto \sum_{i_1, ..., i_p=1}^n a_{i_1, ..., i_p} e^{i_1} \wedge ... \wedge e^{i_p} ,
\end{align*}
where $\{e^1, ..., e^n\}$ is the dual basis of the standard basis of $\mathbb{R}^n$. We will identify both vector spaces without always stating it explicitly.  
\end{bem}

We need to know how $\delta$ acts on $\Omega^p(T^n, \mathbb{C})$ for flat tori $T^n$ with respect to the bases induced by the $dx^i$, $i \in \{1, ..., n\}$. To this end, let $\omega = \sum_{i_1, ..., i_p=1}^n \omega_{i_1...i_p} dx^{i_1} \wedge ... \wedge dx^{i_p} \in \Omega^p(T^n,\mathbb{C})$. Then one can check that
\begin{align} \label{delta2}
\delta \omega = - \sum_{i_1, ..., i_p = 1}^n \sum_{k=1}^p (-1)^{k-1} \partial_{i_k} \omega_{i_1...i_p} dx^{i_1} \wedge ... \wedge \widehat{dx^{ik}} \wedge ... \wedge dx^{i_p} . 
\end{align}

This is also true when $T^n$ and $\{dx^1, ..., dx^n\}$ are replaced by $\mathbb{R}^n$ and the dual basis $\{e^1, ..., e^n\}$ of the global standard basis of $T\mathbb{R}^n$, respectively.

\begin{bem} \label{5}
Let $T^n:= \mathbb{R}^n/\Lambda$ be the flat torus to the lattice $\Lambda$.
\begin{enumerate}
\item  For $l \in \Lambda^*$ the funcions $\chi_l : T^n \rightarrow \mathbb{C} : [x] \mapsto e^{2\pi i l(x)}$ are exactly the characters of the Lie group $T^n$ and form a basis of $L^2(T^n, \mathbb{C})$, according to the Peter-Weyl theorem {\normalfont (\cite[p. 250]{pw})}. \\
The $\chi_l$ are well defined, since for all $l \in \Lambda^*$ the map $\mathbb{R}^n \rightarrow \mathbb{C} : x \mapsto e^{2\pi i l(x)}$ is $\Lambda$-periodic, more precisely: For $\lambda \in \Lambda$, we have 
\begin{align*}
\chi_l([x + \lambda])= e^{2\pi i l(x+\lambda)} = e^{2\pi i l(x)} e^{2 \pi i l(\lambda)} = e^{2 \pi i l(x)} = \chi_l([x]) 
\end{align*} 
since $l(\lambda) \in \mathbb{Z}$.
\item For $l \in \Lambda^*$ the $\chi_l$ (and multiples of it) are exactly the eigenfunctions of $\Delta_0$ on $\mathcal{C}^{\infty}(T^n, \mathbb{C})$ to the eigenvalues $4\pi^2|l|^2$ and we have
\begin{align} \label{Specl0}
\mathrm{Spec}(\Delta_0^{T^n}) = \mathop{\Cup}_{l \in \Lambda^*} \{4\pi^2|l|^2\}
\end{align}
 	as a weighted set.
\end{enumerate}
\end{bem}

\subsection{Eigendecomposition}

In this section, let $\Lambda$ be a lattice in $\mathbb{R}^n$, $n > 0$ and $T^n := \mathbb{R}^n/\Lambda$ the associated flat torus. For $\alpha, \beta > 0$ and $1 \leq p \leq n$ we calculate the eigenvalues of $F_{\alpha \beta,p}^{T^n}$ together with the associated eigenspaces. \\

Let $\omega \in \Omega^p(T^n, \mathbb{C})$. By the isomorphism \ref{tensoriso} and Remark~\ref{5} (i), we can write $\omega = \sum_{l \in \Lambda^*} \chi_l \omega^l$, where $\omega^l \in \Lambda^p (\mathbb{R}^n)^*$. Using \eqref{Formulasddelta}, a straightforward calculation then shows
\begin{align*}
d\delta \omega = 4\pi^2 \sum_{l \in \Lambda^*} \chi_l l \wedge (l^{\sharp} \lrcorner \omega^l)
\end{align*}
and analogously
\begin{align*}
\delta d \omega &= 4\pi^2 \sum_{l \in \Lambda^*} \chi_l \big( - l \wedge (l^{\sharp} \lrcorner \omega^l) + |l|^2 \omega^l \big) .
\end{align*}
Put together, we obtain 
\begin{align*}
F_{\alpha \beta}^{T^n} \omega = 4\pi^2 \sum_{l \in \Lambda^*} \chi_l \Big((\alpha - \beta) l \wedge (l^{\sharp} \lrcorner \omega^l) + \beta |l|^2 \omega^l \Big) .
\end{align*}
Due to Remark \ref{5}, comparison of coefficients gives that the eigenvalue equation 
\begin{align*}
F_{\alpha \beta}^{T^n} \omega = \lambda \omega = \lambda \sum_{l \in \Lambda^*} \chi_l \omega^l 
\end{align*}
is satisfied for a $\lambda \in \mathbb{R}$ if and only if for all $l \in \Lambda^*$, we have
\begin{align*} 
4 \pi^2 \big( (\alpha - \beta) l \wedge (l^{\sharp} \lrcorner \omega^l) + \beta |l|^2 \omega^l \big) = \lambda \omega^l . \ \ \ \ \ \ \tag{$EG_{\lambda, l}$}
\end{align*}

\begin{deff}
For $l \in \Lambda^*$ and $\gamma \in \mathbb{R}$, we set
\begin{align*}
\lambda_{\gamma}^l &:= 4\pi^2\gamma|l|^2
\end{align*}
and
\begin{align*}
V_l^p &:= \chi_l \{ l \wedge \eta \mid \eta \in \Omega^{p-1}_{\mathrm{par}}(T^n) \}, \\
W_l^p &:= \chi_l \{ \omega \in \Omega^p_{\mathrm{par}}(T^n) \mid l^{\sharp} \lrcorner \omega = 0 \} .
\end{align*}
\end{deff}

\begin{bem}
It is obvious that for all $k,l \in \Lambda^*$ with $k \neq l$ the spaces $V_k^p$, $V_l^p$, $W_k^p$ and $W_l^p$ are pairwise orthogonal. Therefore $V_k^p \oplus V_l^p$, $W_k^p \oplus W_l^p$, $V_k^p \oplus W_l^p$ and $V_k^p \oplus W_k^p$ are direct sums.
\end{bem}

Now let $l \in \Lambda^*$. We notice that $l \wedge \eta$, with $\eta \in \Omega^{p-1}_{\mathrm{par}}(T^n)$ and $l^{\sharp} \lrcorner \eta =0$, satisfies the equation $(EG_{\lambda_{\alpha}^l,l})$ and $\omega^l \in \Omega^p_{\mathrm{par}}(T^n)$, with $l^{\sharp} \lrcorner \omega^l = 0$, the equation $(EG_{\lambda_{\beta}^l,l})$. Therefore, we have
\begin{align*}
F_{\alpha \beta}^{T^n} \chi_l l \wedge \eta = \lambda_{\alpha}^l\chi_l l \wedge \eta \ \  \text{and} \ \ \ F_{\alpha \beta}^{T^n} \chi_l\omega^l = \lambda_{\beta}^l \chi_l \omega^l .
\end{align*}
Hence, if $\eta, \omega^l \neq 0$, then $\chi_l l \wedge \eta$ and $\chi_l\omega^l$ are eigenforms of $F_{\alpha \beta}^{T^n}$ to the eigenvalues $\lambda_{\alpha}^l$ and $\lambda_{\beta}^l$, respectively. Because

\begin{align*}
\dim \big( \{ l \wedge \eta \mid \eta \in \Omega^{p-1}_{\mathrm{par}}(T^n, \mathbb{C}) \} \big) &= \binom{n-1}{p-1} \\
\dim \big( \{ \omega \in \Omega^p_{\mathrm{par}}(T^n, \mathbb{C}) \mid l^{\sharp} \lrcorner \omega = 0 \} \big) &= \binom{n-1}{p} ,
\end{align*}
we already know that
\begin{align*}
\left( \mathop{\Cup}_{l \in \Lambda^*} \{\lambda_{\alpha}^l\}_{\binom{n-1}{p-1}} \right) \Cup \left( \mathop{\Cup}_{l \in \Lambda^*} \{\lambda_{\beta}^l\}_{\binom{n-1}{p}} \right) \subseteq \mathrm{Spec}(F_{\alpha \beta, p}^{T^n})
\end{align*}
and that, for all $l \in \Lambda^*$,
\begin{align} \label{vwie}
V_l^p \subseteq \mathrm{Eig}(F_{\alpha \beta, p}^{T^n}, \lambda_{\alpha}^l) \ \ \ \text{and} \ \ \
W_l^p \subseteq \mathrm{Eig}(F_{\alpha \beta, p}^{T^n}, \lambda_{\beta}^l) .
\end{align}

In fact, with $\lambda_{\alpha}^l$ and $\lambda_{\beta}^l$ for $l \in \Lambda^*$, we already found the whole spectrum of $F_{\alpha \beta}^{T^n}$.

\begin{thm} \label{thmspt}
Let $\alpha, \beta>0$ and $1 \leq p \leq n$. The spectrum of the operator $F_{\alpha \beta, p}^{T^n}$ is given by
\begin{align*}
\mathrm{Spec}(F_{\alpha \beta, p}^{T^n}) = \left( \mathop{\Cup}_{l \in \Lambda^*} \{\lambda_{\alpha}^l\}_{\binom{n-1}{p-1}} \right) \Cup \left( \mathop{\Cup}_{l \in \Lambda^*} \{\lambda_{\beta}^l\}_{\binom{n-1}{p}} \right) 
\end{align*}
and the associated eigenspaces are for $k \in \Lambda^*$
\begin{align} \label{eig1}
\mathrm{Eig}(F_{\alpha \beta, p}^{T^n}, \lambda_{\alpha}^k) = \bigoplus_{\substack{l \in \Lambda^*: \\ |l|=|k|}} V_l^p \ \oplus \!\!\! \bigoplus_{\substack{l \in \Lambda^*: \\ |l| = \sqrt{\frac{\alpha}{\beta}}|k|}} \!\!\!\! W_l^p 
\end{align}
and
\begin{align} \label{eig2}
\mathrm{Eig}(F_{\alpha \beta, p}^{T^n}, \lambda_{\beta}^k) = \!\!\!\! \bigoplus_{\substack{l \in \Lambda^*: \\ |l|= \sqrt{\frac{\beta}{\alpha}}|k|}} \!\!\!\! V_l^p \ \oplus \ \bigoplus_{\substack{l \in \Lambda^*: \\ |l| = |k|}} W_l^p . 
\end{align}
\end{thm}

\begin{proof}
By elliptic regularity, any eigenfunction $\omega$ of $F_{\alpha\beta}^{T^n}$ must be smooth so that we can write $\omega = \sum_{l \in \Lambda^*} \chi_l \omega^l$. In order for $\lambda \in \mathbb{R}$ to be an eigenvalue, for each $l \in \Lambda^*$, we get the necessary condition $(EG_{\lambda, l})$. Now because for each $l \in \Lambda$, we have
\begin{equation*}
  \Lambda^p(\mathbb{R}^n)^* = \bigl\{l \wedge \eta \mid \eta \in \Lambda^{p-1}(\mathbb{R}^n)^*\bigr\}  \oplus \bigl\{ \omega \in \Lambda^p (\mathbb{R}^n)^* \mid l^\sharp \lrcorner \omega = 0 \bigr\},
\end{equation*}
the statement follows with a view on Remark~\ref{RemarkIsoFormsOnRn}.
\end{proof}

\begin{bem} \label{bemnt1} $ $
\begin{enumerate}
\item We can express the spectrum of $F_{\alpha \beta, p}^{T^n}$ in terms of the spectrum of $\Delta_0^{T^n}$:
\begin{align*}
\mathrm{Spec}(F_{\alpha \beta, p}^{T^n}) &= \left( \mathop{\Cup}_{l \in \Lambda^*} \{\lambda_{\alpha}^l\}_{\binom{n-1}{p-1}} \right) \Cup \left( \mathop{\Cup}_{l \in \Lambda^*} \{\lambda_{\beta}^l\}_{\binom{n-1}{p}} \right) = \left( \mathop{\Cup}_{l \in \Lambda^*} \{ \lambda_{\alpha}^l \}_1 \right) \ \! ^{\binom{n-1}{p-1}}\Cup \ \! \! ^{\binom{n-1}{p}} \left( \mathop{\Cup}_{l \in \Lambda^*} \{\lambda_{\beta}^l\}_1 \right) \\
&= \alpha \left( \mathop{\Cup}_{l \in \Lambda^*} \{\lambda_1^l\}_1 \right)  \ \! ^{\binom{n-1}{p-1}}\Cup \ \! \! ^{\binom{n-1}{p}} \beta \left( \mathop{\Cup}_{l \in \Lambda^*} \{\lambda_1^l\}_1 \right) \\ 
&\!\stackrel{\eqref{Specl0}}{=} \alpha \cdot \mathrm{Spec}(\Delta_0^{T^n})  \ \! ^{\binom{n-1}{p-1}}\Cup \ \! \! ^{\binom{n-1}{p}} \beta \cdot \mathrm{Spec}(\Delta_0^{T^n}) .
\end{align*}
\item We have $\mathrm{Eig}(F_{\alpha \beta, p}^{T^n}, 0) = \Omega^p_{\mathrm{par}}(T^n)$. 
\item For $k \in \Lambda^*$, we have
\begin{align*}
\mathrm{Eig}(F_{\alpha \alpha, p}^{T^n}, \lambda_{\alpha}^k) = \bigoplus_{\substack{l \in \Lambda^*: \\ |l|=|k|}}(V_l^p \oplus W_l^p) = \bigoplus_{\substack{l \in \Lambda^*: \\ |l|=|k|}} \chi_l \Omega^p_{\mathrm{par}}(T^n).
\end{align*} 
\item Let the dual lattice $\Lambda^*$ now be rational, i.e.\ $|l|^2 \in \mathbb{Q}$ for all $l \in \Lambda^*$. Then for a generic choice of $\alpha$ and $\beta$ (i.e.\ if for example they are chosen at random from a finite interval with uniform distribution) the second summand of \eqref{eig1} and first summand of \eqref{eig2} disappear for $k \neq 0$: 
\begin{align*}
\mathrm{Eig}(F_{\alpha \beta, p}^{T^n}, \lambda_{\alpha}^k) = \bigoplus_{\substack{l \in \Lambda^*: \\ |l|=|k|}} V_l^p \ \ \ \text{and} \ \ \
\mathrm{Eig}(F_{\alpha \beta, p}^{T^n}, \lambda_{\beta}^k) = \bigoplus_{\substack{l \in \Lambda^*: \\ |l|=|k|}} W_l^p ,
\end{align*}
because generically $|l|^2= \frac{\alpha}{\beta}|k|^2$ and $|l|^2=\frac{\beta}{\alpha}|k|^2$ are satisfied for no $l \in \Lambda^*$. For instance, this happens for $\frac{\alpha}{\beta}$ irrational. In this case, the eigenvalues $\lambda_{\alpha}^k$ and $\lambda_{\beta}^l$ are different for all $k,l \in \Lambda^*$. \\
However, in the non-generic case that $\frac{\alpha}{\beta} \in \mathbb{Q}$, there can exist $k, l \in \Lambda^*$ such that $\lambda_{\alpha}^k = \lambda_{\beta}^l$. Consequently, in this case, ``mixed''\ eigenspaces can occur. 
\end{enumerate}
\end{bem}

\subsection{Multiplicities}
 
\begin{deff}
Let $\Lambda$ be a lattice. For $r \in \mathbb{R}_{\geq 0}$ we set
\begin{align*}
A_{\Lambda}(r) := \# \{l \in \Lambda^* \mid |l| = r \}.
\end{align*} 
Hereafter, we simply write $A$ instead of $A_{\Lambda}$ if it is clear which lattice is meant.
\end{deff}

\begin{bem}
In the case that the dual lattice $\Lambda^*$ is generated by an orthonormal basis, finding the value of $N(R) := \sum_{r \leq R} A(r)$ for $n=2$ is just the Gauss circle problem {\normalfont (\cite{hardy})}.
\end{bem}

From \eqref{eig1} and \eqref{eig2} we can directly read off the geometric multiplicities of the eigenvalues:

\begin{kor} \label{gv}
Let $\alpha, \beta >0$, $1 \leq p \leq n$ and $k \in \Lambda^*$. Then we have
\begin{align*}
\mathrm{dim} \big( \mathrm{Eig}(F_{\alpha \beta,p}^{T^n}, \lambda_{\alpha}^k) \big) = \binom{n-1}{p-1} A(|k|) + \binom{n-1}{p} A\left(\sqrt{\frac{\alpha}{\beta}}|k|\right)
\end{align*}
and 
\begin{align*}
\mathrm{dim} \big( \mathrm{Eig}(F_{\alpha \beta,p}^{T^n}, \lambda_{\beta}^k) \big) = \binom{n-1}{p} A(|k|) + \binom{n-1}{p-1} A\left(\sqrt{\frac{\beta}{\alpha}}|k|\right) .
\end{align*}
\end{kor}

\subsection{Isospectrality}

Subsequently, we investigate under which circumstances the operators $F_{\alpha \beta, p}$ for $\alpha, \beta > 0$ and $1 \leq p \leq n$ can have the same spectrum on different flat tori.

\subsubsection{Isometric flat tori}

\begin{bem}
Let $\Lambda_1$ and $\Lambda_2$ be two lattices in $\mathbb{R}^n$. Then $\mathbb{R}^n / \Lambda_1$ and $\mathbb{R}^n / \Lambda_2$ are isometric if and only if there is a linear isometry $I: \mathbb{R}^n \rightarrow \mathbb{R}^n$ with $I(\Lambda_1) = \Lambda_2$.
\end{bem}

If $M$ and $N$ are compact isometric Riemannian manifolds, the spectra of $F_{\alpha \beta, p}^M$ and $F_{\alpha \beta, p}^N$ coincide. In particular, two isometric flat tori have the same spectra with respect to $F_{\alpha \beta,p}$.
Now one faces the question whether the converse of this statement is true, i.e.\ whether the spectrum of $F_{\alpha \beta,p}^{T^n}$ already determines the flat tori $T^n$ up to isometry. 

\medskip
In the case $n=1$ the answer is found to be ``yes'', because the lattices then have the form $r \mathbb{Z}$ for $r \in \mathbb{R}\setminus \{0\}$ and the spectrum of the associated flat torus $\mathbb{R}/r\mathbb{Z}$, which is isomorphic to the one-dimensional sphere $S^1_{\frac{r}{2\pi}}$ with radius $\frac{r}{2\pi}$, is:
\begin{align*}
\mathrm{Spec}(F_{\alpha \beta, p}^{\mathbb{R} / r\mathbb{Z}}) = \mathop{\Cup}_{l \in \mathbb{Z}} \left\{4\pi^2\alpha \frac{l^2}{r^2}\right\}_{\binom{n-1}{p-1}} .
\end{align*}
So if $\mathbb{R} / r\mathbb{Z}$ and $\mathbb{R} / r' \mathbb{Z}$ for $r, r' \in \mathbb{R} \setminus \{0\}$ have the same spectrum, then $r = \pm r'$, i.e.\ both tori are identical, thus trivially isometric. 

\medskip
To give an answer to the above question for further dimensions $n$, first we show the following useful Lemma \ref{1}.

\begin{deff} Let $\mathcal{M} := \{ M \subseteq \mathbb{R} \mid M \text{ is a discrete weighted set bounded from below}\}$.
\end{deff}

\begin{lem} \label{1}
Let $\alpha, \beta > 0$, $l, m \in \mathbb{N}$ and $A, B \in \mathcal{M}$ with $\alpha A \ ^l \! \Cup^m \beta A = \alpha B \ ^l \! \Cup^m \beta B$. Then $A = B$.
\end{lem}

\begin{proof}
First suppose $\alpha \leq \beta$. We consider the map 
\begin{align*}
f_{\alpha \beta}: \mathcal{M} \rightarrow \mathcal{M}: C \mapsto \alpha C \ ^l \! \Cup^m \beta C.
\end{align*} 
and show that it is injective, i.e.\ we show that for $M \in f_{\alpha \beta}(\mathcal{M}) \subseteq \mathcal{M}$ there exists an unique $C \in \mathcal{M}$ such that $M = \alpha C \ ^l \! \Cup^m \beta C$. This can be seen from the following reconstruction: \\
We set $M_0 := M$ and for $k \in \mathbb{N}_0$ iteratively $\lambda_k := \frac{1}{\alpha} \mathrm{min} (M_k)$ and $M_{k+1}:= M_k \setminus \{ (\alpha \lambda_k,l), (\beta \lambda_k,m)\}$. (The minima exist because the sets $M_k$ are discrete and bounded from below as subsets of $M$.) Then $M = \alpha C \ ^l \! \Cup^m \beta C$ with $C := \Cup_{i \in \mathbb{N}_0} \{ \lambda_i \}_1 $. By construction, $C$ is unique and $f_{\alpha \beta}$ therefore injective. \\
Thus, due to $f_{\alpha \beta}(A) = \alpha A \ ^l \! \Cup^m \beta A = \alpha B \ ^l \! \Cup^m \beta B = f_{\alpha \beta}(B)$, we have $A=B$. \\
The case $\alpha > \beta$ is shown analogously by setting $\lambda_k := \frac{1}{\beta} \mathrm{min}(M_k)$ in the above proof.
\end{proof}

\begin{bem}
Without the concept of the weighted set and union, an ordinary set $C$ can just be reconstructed from the set $\alpha C \cup \beta C$ with the method from the proof if $\alpha C \cap \beta C = \emptyset$.
\end{bem}

\begin{thm} \label{2}
Let $\Lambda_1$ and $\Lambda_2$ be two lattices in $\mathbb{R}^n$, $\alpha, \beta > 0$ and $1 \leq p \leq n$. Then
$F_{\alpha \beta, p}^{\mathbb{R}^n / \Lambda_1}$ and $F_{\alpha \beta, p}^{\mathbb{R}^n / \Lambda_2}$ are isospectral if and only if $\Delta_0^{\mathbb{R}^n / \Lambda_1}$ and $\Delta_0^{\mathbb{R}^n / \Lambda_2}$ are isospectral.
\end{thm}

\begin{proof}
``$\Rightarrow$'': For $i=1,2$ we have seen in \eqref{Specl0}, that on functions, 
\begin{align*}
\mathrm{Spec}(\Delta_0^{\mathbb{R}^n / \Lambda_i}) = \mathop{\Cup}_{l \in \Lambda_i^*} \{ \lambda_1^l \}_1  .
\end{align*}
The eigenvalues are non-negative and thus the spectra are weighted subsets of $\mathbb{R}$ which are bounded from below and discrete due to $\Lambda_i^* \cong \mathbb{Z}^n$. Furthermore, Remark \ref{bemnt1} and the assumption imply
\begin{align*}
\alpha \mathrm{Spec}(\Delta_0^{\mathbb{R}^n / \Lambda_1}) \ ^{\binom{n-1}{p-1}} \! \Cup \ \!\!\! ^{\binom{n-1}{p}} \beta \mathrm{Spec}(\Delta_0^{\mathbb{R}^n / \Lambda_1}) &= \mathrm{Spec}(F_{\alpha \beta, p}^{\mathbb{R}^n / \Lambda_1}) = \mathrm{Spec}(F_{\alpha \beta, p}^{\mathbb{R}^n / \Lambda_2}) \\
&= \alpha \mathrm{Spec}(\Delta_0^{\mathbb{R}^n / \Lambda_2}) \ ^{\binom{n-1}{p-1}} \! \Cup \ \!\!\!^{\binom{n-1}{p}} \beta \mathrm{Spec}(\Delta_0^{\mathbb{R}^n / \Lambda_2}) .
\end{align*}
Lemma \ref{1} therefore yields
$\mathrm{Spec}(\Delta_0^{\mathbb{R}^n / \Lambda_1}) = \mathrm{Spec}(\Delta_0^{\mathbb{R}^n / \Lambda_2})$. \\
``$\Leftarrow$'': If $\Delta_0^{\mathbb{R}^n / \Lambda_1}$ and $\Delta_0^{\mathbb{R}^n / \Lambda_2}$ are isospectral, it follows immediately that
\begin{align*}
\mathrm{Spec}(F_{\alpha \beta, p}^{\mathbb{R}^n / \Lambda_1}) &= \alpha \mathrm{Spec}(\Delta_0^{\mathbb{R}^n / \Lambda_1}) \ ^{\binom{n-1}{p-1}} \! \Cup \ \!\!\!^{\binom{n-1}{p}} \beta \mathrm{Spec}(\Delta_0^{\mathbb{R}^n / \Lambda_1}) \\
&= \alpha \mathrm{Spec}(\Delta_0^{\mathbb{R}^n / \Lambda_2}) \ ^{\binom{n-1}{p-1}} \! \Cup \ \!\!\!^{\binom{n-1}{p}} \beta \mathrm{Spec}(\Delta_0^{\mathbb{R}^n / \Lambda_2}) = \mathrm{Spec}(F_{\alpha \beta, p}^{\mathbb{R}^n / \Lambda_2}) .\qedhere
\end{align*}
\end{proof}

Using Theorem \ref{2}, we can now give an answer to the above asked question for the dimensions $n=2$ and $n \geq 4$, because it is already known that the following results are true for Laplace-Beltrami operators on functions.

\begin{satz} \label{PropOpIsospToriIsom}
Let $\Lambda_1$ and $\Lambda_2$ be two lattices in $\mathbb{R}^n$, $n=2, 3$, $\alpha, \beta > 0$ and $1 \leq p \leq n$. Then
if $F_{\alpha \beta, p}^{\mathbb{R}^2/\Lambda_1}$ and $F_{\alpha \beta, p}^{\mathbb{R}^2/\Lambda_2}$ are isospectral, the flat tori $\mathbb{R}^2/\Lambda_1$ and $\mathbb{R}^2/\Lambda_2$ are isometric.
\end{satz}

\begin{proof}
Because $F_{\alpha \beta, p}^{\mathbb{R}^2/\Lambda_1}$ and $F_{\alpha \beta, p}^{\mathbb{R}^2/\Lambda_2}$ are isospectral, by Theorem \ref{2} this is also true for $\Delta_0^{\mathbb{R}^2 / \Lambda_1}$ and $\Delta_0^{\mathbb{R}^2 / \Lambda_2}$. In dimension $n=2$, the claim then follows from \cite[Proposition B.II.5]{bgm}; in dimension $n=3$, the claim follows from the work of Schiemann \cite{Schiemann}.
\end{proof}

\begin{satz} \label{SatzIsoNonIso}
Let $\alpha, \beta > 0$, $1 \leq p \leq n$ and $n \geq 4$. Then there exist two lattices $\Lambda_1$ and $\Lambda_2$ in $\mathbb{R}^n$, such that $F_{\alpha \beta, p}^{\mathbb{R}^n/\Lambda_1}$ and $F_{\alpha \beta, p}^{\mathbb{R}^n/\Lambda_2}$ are isospectral, whereas the flat tori $\mathbb{R}^n/\Lambda_1$ and $\mathbb{R}^n/\Lambda_2$ are not isometric.
\end{satz}

\begin{proof}
In the paper of J. H. Conway and N. J. A. Sloane (\cite{conwaysloane}), it is shown for the dimension $n=4$ that there are lattices $\Lambda_1$ and $\Lambda_2$ in $\mathbb{R}^n$ such that $\Delta_0^{\mathbb{R}^n/\Lambda_1}$ and $\Delta_0^{\mathbb{R}^n/\Lambda_2}$ are isospectral, but $\mathbb{R}^n/\Lambda_1$ and $\mathbb{R}^n/\Lambda_2$ are not isometric. Due to \cite[chapter III, Proposition B.III.1]{bgm} therefore there are such in every dimension $n \geq 4$. The result follows with Theorem \ref{2}.
\end{proof}

For the dimension $n=3$ this question still seems to be open.

\subsubsection{Variation of parameters}

For flat tori $T_1$ and $T_2$, $\alpha, \beta > 0$ and $1 \leq p \leq n$ we have shown that if $F_{\alpha \beta, p}^{T_1}$ and $F_{\alpha \beta, p}^{T_2}$ are isospectral, we can conclude that $T_1$ and $T_2$ are isometric in dimensions $1$ and $2$, while from dimension $4$ onward, there are cases in which they are not isometric. Now we consider the case that $F_{\alpha \beta, p}^{T_1}$ and $F_{\alpha' \beta', p}^{T_2}$ are isospectral for $( \alpha, \beta) \neq (\alpha', \beta')$ and $n \neq 2p$, respectively $\{\alpha,\beta\} \neq \{\alpha', \beta'\}$ and $n=2p$ (see Remark \ref{2d}), where $\alpha', \beta' > 0$ as well.

\begin{bem} \label{2d}
In Proposition \ref{symn=2} we already recognized that for $n$-dimensional flat tori $T^n$ with $n = 2p$ the spectrum of $F_{\alpha \beta, p}^{T^n}$ is symmetric in $\alpha$ and $\beta$:
\begin{align*}
\mathrm{Spec}(F_{\alpha \beta, p}^{T^n}) = \alpha \mathrm{Spec}(\Delta_0^{T^n}) \ ^{\binom{2p-1}{p}} \! \Cup \ \!\!\!^{\binom{2p-1}{p}} \beta \mathrm{Spec}(\Delta_0^{T^n}) = \mathrm{Spec}(F_{\beta \alpha, p}^{T^n}) .
\end{align*}
\end{bem}

\begin{thm} \label{4}
Let $T^n$ be a flat torus, $\alpha, \alpha', \beta, \beta' > 0$, $1 \leq p \leq n$ and $n \neq 2p$. Then $F_{\alpha \beta,p}^{T^n}$ and $F_{\alpha' \beta',p}^{T^n}$ are isospectral if and only if $( \alpha, \beta) = ( \alpha', \beta')$. \\
For $n=2p$ the statement holds with $\{\alpha, \beta\} = \{\alpha', \beta'\}$ instead of $( \alpha, \beta) = ( \alpha', \beta')$.
\end{thm}

\begin{proof}
``$\Rightarrow$'': Let $\Lambda$ be a lattice in $\mathbb{R}^n$ with $T^n := \mathbb{R}^n / \Lambda$ and $A:= \mathrm{Spec}(\Delta_0^{T^n})$. \\
Let first $n \neq 2p$. We consider the map 
\begin{align*}
f_A: (0, \infty) \times (0, \infty) &\rightarrow \mathcal{M} \\
(\gamma, \delta) &\mapsto \gamma A \ ^{\binom{n-1}{p-1}} \! \Cup \ \!\!\!^{\binom{n-1}{p}} \delta A  = \mathrm{Spec}(F_{\gamma \delta,p}^{T^n})
\end{align*} 
and show that it is injective. For $M$ in the image of $f_A$, we obtain a unique preimage under $f_A$ in the following way:  \\
Let $\widetilde{M} := M \setminus \{ (0,\binom{n}{p}) \}$. The minimum $\widetilde{m} := \min(\widetilde{M})$ is positive because $0$ has the multiplicity $\binom{n}{p}$. Now we choose $k \in \Lambda^*$ in such a way that 
\begin{align*}
|k| = \min_{l \in \Lambda^*\setminus \{0\}} |l| . 
\end{align*}
Then $\lambda_1^k$ is the smallest positive element of $A$, $A(|k|) \neq 0$ and we have $A(r) = 0$ for all $0 < r < |k|$. Hence due to Corollary \ref{gv}, the multiplicity $\mathrm{mult}(\widetilde{m})$ of $\widetilde{m}$ is either $\binom{n-1}{p-1} A(|k|)$, $\binom{n-1}{p}A(|k|)$ or $\binom{n}{p}A(|k|)$. 
\begin{enumerate}
\item[$\bullet$] In the case that $\mathrm{mult}(\widetilde{m}) = \binom{n-1}{p-1} A(|k|)$, we set $\gamma := \frac{\widetilde{m}}{\lambda_1^k}$. Let now $M' := M \setminus \Cup_{i=1}^{\binom{n-1}{p-1}} \gamma A$ and $\widetilde{M}' := M' \setminus \{ (0, \binom{n-1}{p}) \}$. Then $\widetilde{m}' := \min{\widetilde{M}'} > 0$ and we put $\delta := \frac{\widetilde{m}'}{\lambda_1^k}$. 
\item[$\bullet$] If $\mathrm{mult}(\widetilde{m}) = \binom{n-1}{p}A(|k|)$, we set $\delta := \frac{\widetilde{m}}{\lambda_1^k}$.  Let now $M' := M \setminus \Cup_{i=1}^{\binom{n-1}{p}} \delta A$ and $\widetilde{M}' := M' \setminus \{(0,\binom{n-1}{p-1})\}$. Then $\widetilde{m}' := \min(\widetilde{M}') > 0$ and we define $\gamma := \frac{\widetilde{m}'}{\lambda_1^k}$. 
\item[$\bullet$] In the last case that $\mathrm{mult}(\widetilde{m}) = \binom{n}{p} A(|k|)$, we put $\gamma := \delta := \frac{\widetilde{m}}{\lambda_1^k}$. 
\end{enumerate}
In each case, $M = \gamma A \ ^{\binom{n-1}{p-1}} \! \Cup^{\binom{n-1}{p}} \delta A$. Here $\gamma$ and $\delta$ are unique by construction and therefore $f_A$ is injective. \\
By assumption, $f_A(\alpha, \beta) = \mathrm{Spec}(F_{\alpha \beta,p}^{T^n}) = \mathrm{Spec}(F_{\alpha' \beta',p}^{T^n}) = f_A(\alpha', \beta')$, thus $(\alpha, \beta) = (\alpha', \beta')$.

\medskip
Now let $n=2p$. We show that the map
\begin{align*}
\widetilde{f}_A: \big\{ C \subset (0,\infty) \mid \# C \in \{1,2\} \big\} &\rightarrow \mathcal{M}  \\
\{\gamma, \delta\} &\mapsto \gamma A \ ^{\binom{2p-1}{p}} \! \Cup \ \!\!\!^{\binom{2p-1}{p}} \delta A = \mathrm{Spec}(F_{\gamma \delta,p}^{T^{2p}}) = \mathrm{Spec}(F_{\delta \gamma,p}^{T^{2p}})
\end{align*}
is injective. To this end, let $M$ be in the image of $f_A$, $\widetilde{m} > 0$ and $k \in \Lambda^*$ as above. We set $\gamma := \frac{\widetilde{m}}{\lambda_1^k}$. Now let $M' := M \setminus \Cup_{i=1}^{\binom{2p-1}{p}} \gamma A$ and $\widetilde{M}' := M' \setminus \{ (0,\binom{2p-1}{p}) \}$. We have $\widetilde{m}' := \min(\widetilde{M}') > 0$ and we define $\delta := \frac{\widetilde{m}'}{\lambda_1^k}$. Then $M = \gamma A \ ^{\binom{2p-1}{p}} \! \Cup \ \!\!\!^{\binom{2p-1}{p}} \delta A$, where $\{ \gamma, \delta \}$ is unique by construction. Thus $\widetilde{f}_A$ is injective. With $\widetilde{f}_A(\{\alpha, \beta\}) = \mathrm{Spec}(F_{\alpha \beta,p}^{T^{2p}}) = \mathrm{Spec}(F_{\alpha' \beta',p}^{T^{2p}}) = \widetilde{f}_A(\{\alpha', \beta'\})$ we obtain that $\{\alpha, \beta\} = \{\alpha', \beta'\}$.

\medskip
``$\Leftarrow$'': The converse implications are trivial.
\end{proof}

The next corollary shows that there are $\alpha, \alpha', \beta, \beta' > 0$ and lattices $\Lambda, \Lambda'$ in $\mathbb{R}^n$ such that $F_{\alpha \beta, p}^{\mathbb{R}^n/\Lambda}$ and $F_{\alpha' \beta', p}^{\mathbb{R}^n/\Lambda'}$ are isospectral. 

\begin{kor}
Let $\Lambda$ be a lattice in $\mathbb{R}^n$, $\alpha, \alpha', \beta, \beta' > 0$ and $c \in \mathbb{R} \setminus \{0\}$. Then for $n \neq 2p$: \\ 
$F_{\alpha \beta, p}^{\mathbb{R}^n/\Lambda}$ and $F_{\alpha' \beta', p}^{\mathbb{R}^n/c\Lambda}$ are isospectral if and only if  $(\alpha', \beta') = (c^2\alpha, c^2 \beta)$. \\
For $n=2p$ the statement holds with $\{\alpha', \beta'\} = \{c^2\alpha, c^2 \beta \}$ instead of $(\alpha', \beta') = (c^2\alpha, c^2 \beta)$.
\end{kor}

\begin{proof}
For the spectrum of $F_{\alpha' \beta', p}^{\mathbb{R}^n/c\Lambda}$, we have
\begin{align*}
\mathrm{Spec}\left(F_{\alpha' \beta', p}^{\mathbb{R}^n/c\Lambda}\right) &= \alpha' \left( \mathop{\Cup}_{l \in \frac{1}{c}\Lambda^*} \{4\pi^2|l|^2\}_1 \right) \ ^{\binom{n-1}{p-1}} \! \Cup \ \!\!\!^{\binom{n-1}{p}} \beta' \left( \mathop{\Cup}_{l \in \frac{1}{c}\Lambda^*} \{4\pi^2|l|^2\}_1 \right) \\
&= \frac{\alpha'}{c^2} \left( \mathop{\Cup}_{l \in \Lambda^*} \{4\pi^2|l|^2\}_1 \right) \ ^{\binom{n-1}{p-1}} \! \Cup \ \!\!\!^{\binom{n-1}{p}} \frac{\beta'}{c^2} \left( \mathop{\Cup}_{l \in \Lambda^*} \{4\pi^2|l|^2\}_1 \right) \\
&= \mathrm{Spec}\left(F_{\frac{\alpha'}{c^2} \frac{\beta'}{c^2}, p}^{\mathbb{R}^n/\Lambda}\right) .
\end{align*}
Theorem \ref{4} therefore implies that $F_{\alpha \beta, p}^{\mathbb{R}^n/\Lambda}$ and $F_{\alpha' \beta', p}^{\mathbb{R}^n/c\Lambda}$ are isospectral in dimension $n \neq 2p$ if and only if $(\alpha', \beta') = (c^2\alpha, c^2 \beta)$ and in dimension $n=2p$ if and only if $\{\alpha', \beta'\} = \{c^2\alpha, c^2 \beta \}$.
\end{proof}

\begin{bem}
Given two $n$-dimensional flat tori $T$ and $T^\prime$ together with parameters $\alpha, \beta, \alpha^\prime, \beta^\prime>0$, one can consider the operators $F_{\alpha\beta, p}^T$ and $F_{\alpha^\prime\beta^\prime, p}^{T^\prime}$. For $\alpha = \alpha^\prime$, $\beta = \beta^\prime$, it is clear from Prop.~\ref{SatzIsoNonIso} that in dimensions $n \geq 4$, one can choose non-isometric tori $T$, $T^\prime$ in such a way that $F_{\alpha\beta, p}^T$ and $F_{\alpha^\prime\beta^\prime, p}^{T^\prime}$ are isospectral. The author expects that in a similar way, one can find an example of non-isometric tori  $T$, $T^\prime$ such that $F_{\alpha\beta, p}^T$ and $F_{\alpha^\prime\beta^\prime, p}^{T^\prime}$ are isospectral for $\{\alpha, \beta\} \neq \{\alpha^\prime, \beta^\prime\}$. 

On the other hand, the following question is still open: In dimension $n=2$ or $3$, does the isospectrality of $F_{\alpha\beta, p}^T$ and $F_{\alpha^\prime\beta^\prime, p}^{T^\prime}$ already imply that $T$ is isometric to $T^\prime$ and $\alpha = \alpha^\prime$, $\beta = \beta^\prime$ (respectively $\{\alpha, \beta\} = \{\alpha^\prime, \beta^\prime\}$ if $n=2$ and $p=1$)?
\end{bem}


\section{Spectrum on round spheres}

Now we consider for $n \in \mathbb{N}$ the $n$-dimensional unit sphere $(S^n, g) \subset (\mathbb{R}^{n+1}, g_{\mathrm{std}})$, which is embedded in $\mathbb{R}^{n+1}$ via the canonical inclusion $\iota: S^n \rightarrow \mathbb{R}^{n+1}$ and endowed with the metric $g := \iota^*g_{\mathrm{std}}$ induced by the standard metric $g_{\mathrm{std}}(\cdot,\cdot):= \langle \cdot,\cdot \rangle$ on $\mathbb{R}^{n+1}$. \\
In the following, let $\alpha, \beta > 0$ and $1 \leq p \leq n$. \\

For vector fields $X \in \Gamma(\mathbb{R}^{n+1}, T\mathbb{R}^{n+1})$ on $\mathbb{R}^{n+1}$ we denote by $\widetilde{X} := X\big{|}_{S^n}$ the restriction of $X$ to $S^n$. Note that this is a vector field on $S^n$ if and only if $X$ is tangent to $S^n$. From now on, let $\nabla$ be the Levi-Civita connection on $\mathbb{R}^{n+1}$ and $\widetilde{\nabla}$ the one on $S^n$. The outward directed normal vector field we call 
\begin{align*}
\nu: \mathbb{R}^{n+1} \setminus \{0\} \rightarrow T\mathbb{R}^{n+1}: x \mapsto \frac{x}{\| x \|} .
\end{align*}

Its covariant derivative with respect to $X \in \Gamma(\mathbb{R}^{n+1}, T\mathbb{R}^{n+1})$ is
\begin{align}  \label{nxn}
\nabla_X \nu = \frac{1}{r} \big(X - \langle X, \nu \rangle \nu \big).
\end{align}
Here, $r : \mathbb{R}^{n+1} \rightarrow \mathbb{R}_{\geq 0}$ is defined by $x \mapsto r(x) := \|x\|$. \\
The connections $\nabla$ and $\widetilde{\nabla}$ are related as follows:
For $X, Y \in \Gamma(\mathbb{R}^{n+1}, T\mathbb{R}^{n+1})$ tangent to $S^n$, we have
\begin{align} \label{nnt}
\widetilde{\nabla_X Y} &= \widetilde{\nabla}_{\widetilde{X}} \widetilde{Y} - \langle \widetilde{X}, \widetilde{Y} \rangle \widetilde{\nu} .
\end{align}

On can easily convince oneself of the fact that
\begin{align}  \label{ivt}
\iota^*(\omega(X_1, ..., X_p)) = (\iota^*\omega)(\widetilde{X_1}, ..., \widetilde{X_p})
\end{align}
for all $\omega \in \Omega^p(\mathbb{R}^{n+1})$ and all $X_1, ..., X_p \in \Gamma(\mathbb{R}^{n+1},T\mathbb{R}^{n+1})$ tangent to $S^n$, since $d\iota(\widetilde{X_i}) = \widetilde{X_i}$ for $i \in \{1, ..., p\}$. \\

The exterior derivative $d$ is natural, i.e.\ commutes with the pullback $f^*$ along differentiable maps $f$.
In general, this is not true for the co-differential $\delta$ instead of $d$. In the following lemma, we investigate in which way $\delta$ commutes with the pullback $\iota^*: \Omega^*(\mathbb{R}^{n+1}) \rightarrow \Omega^*(S^n)$.

\begin{lem} \label{id}
Let $\omega \in \Omega^p(\mathbb{R}^{n+1})$. Then
\begin{align*}
\iota^*\delta^{\mathbb{R}^{n+1}}\omega &= \delta^{S^n}\iota^*\omega - \iota^* \big( (n-p+1) \cdot (\nu \lrcorner \omega) + \nu \lrcorner  \nabla_{\nu} \omega \big).
\end{align*}
\end{lem}

\begin{proof}
Using the naturality of $d$ and formulas \eqref{nnt} and \eqref{ivt}, it is straightforward to show that
\begin{align*}
\iota^*\nabla_X \omega = \widetilde{\nabla}_{\widetilde{X}} \iota^*\omega + \iota^* \big(X^{\flat} \wedge (\nu \lrcorner \omega)  \big),
\end{align*}
where $\flat$ denotes the inverse of the musical isomorphism $\sharp$. \\
Let $U \subseteq S^n$ be open and $\{\tilde{e_1}, ..., \tilde{e_n}\} \subset \Gamma(U,TS^n)$ a local orthonormal basis of $TS^n$. We set $V:=\{y \in \mathbb{R}^{n+1} \mid \frac{y}{\|y\|} \in U\}$. For $i \in \{1, ...,n\}$, let $e_i \in \Gamma(V,T\mathbb{R}^{n+1})$ be the radial constant extension of $\tilde{e_i}$ on $V \subset \mathbb{R}^{n+1}$, i.e.\ defined via $e_i(y) := \tilde{e_i}(\frac{y}{\|y\|})$ for $y \in V$. Then $\{e_1, ..., e_n, \nu\}$ is a local orthonormal basis of $T\mathbb{R}^{n+1}$. 

Now with a view on \eqref{Formulasddelta}, one can use the above calculation to obtain
\begin{align*}
\iota^*\delta^{\mathbb{R}^{n+1}}\omega &= - \iota^* \left( \sum_{i=1}^{n} e_i \lrcorner \nabla_{e_i}\omega \right) - \iota^* \big( \nu \lrcorner  \nabla_{\nu}\omega \big) \\
&= - \sum_{i=1}^n \widetilde{e_i}  \lrcorner \widetilde{\nabla}_{\widetilde{e_i}} \iota^*\omega - \sum_{i=1}^n  \widetilde{e_i}  \lrcorner  \iota^*\big(e^i \wedge (\nu \lrcorner \omega)\big) - \iota^* (\nu \lrcorner \nabla_{\nu}\omega) \\
&= \delta^{S^n}\iota^*\omega - \iota^* \big( (n - p + 1) \cdot (\nu \lrcorner \omega) + \nu \lrcorner \nabla_{\nu}\omega \big). \qedhere
\end{align*} 
\end{proof}

\subsection{Eigendecomposition}

To determine the spectrum $\mathrm{Spec}(F_{\alpha \beta}^{S^n})$, we first investigate the relation between the operators $F_{\alpha \beta}^{\mathbb{R}^{n+1}}$ and $F_{\alpha \beta}^{S^n}$.

\begin{prop} \label{clu}
For $\omega \in \Omega^p(\mathbb{R}^{n+1})$, we have
\begin{align*}
\iota^* F_{\alpha \beta}^{\mathbb{R}^{n+1}}\omega  =& ~~ F_{\alpha \beta}^{S^n} \iota^* \omega + \alpha \iota^*\Big( (p-n-1) d\big(\nu \lrcorner \omega \big) - d\big(\nu \lrcorner \nabla_{\nu}\omega \big) \Big)  \\ &+ \beta \iota^* \Big((n-p-1) d\big(\nu \lrcorner \omega \big) + d\big(\nu \lrcorner \nabla_{\nu}\omega \big) + p(p-n+1)\omega - n \nabla_{\nu}\omega - \nabla_{\nu}\nabla_{\nu}\omega \Big) .
\end{align*}
\end{prop}

\begin{proof}
Due to the naturality of $d$ and Lemma \ref{id}, we see that
\begin{align} \label{1tegl}
\iota^*d\delta^{\mathbb{R}^{n+1}} \omega = d\iota^*\delta^{\mathbb{R}^{n+1}} \omega = d\delta^{S^n} \iota^*\omega - \iota^* \big( (n-p+1) d(\nu \lrcorner \omega) + d(\nu \lrcorner \nabla_{\nu}\omega) \big),
\end{align}
which is the first term, and
\begin{align} \label{ivmddm}
\iota^*\delta^{\mathbb{R}^{n+1}}d\omega &= \delta^{S^n}d\iota^*\omega - \iota^* \big( (n-p) \cdot (\nu \lrcorner d\omega) + \nu \lrcorner \nabla_{\nu}d\omega \big) .
\end{align}
Now we rewrite the last two terms of \eqref{ivmddm}. A straightforward calculation shows that
\begin{align*}
\nu \lrcorner d\omega &= \sum_{i=1}^n \left( \frac{p}{r}\omega - d(\nu \lrcorner \omega) + \nabla_{\nu}\omega - \nu^{\flat} \wedge (\nu \lrcorner \nabla_{\nu}\omega)  \right),
\end{align*}
hence
\begin{align} \label{ngl1}
\iota^*( \nu \lrcorner d\omega ) = \iota^* \big( p \omega-d(\nu \lrcorner \omega) + \nabla_{\nu}\omega \big) .
\end{align}
Now let $\{e_1, ..., e_n, e_{n+1}:=\nu\}$ be a local orthonormal basis of $T\mathbb{R}^{n+1}$ as in the proof of Lemma \ref{id}. We notice that $\nabla_{\nu}e_i=0$ for $i \in \{1, ..., n+1\}$, which is used to calculate
\begin{align} \label{nwoei}
\nu \lrcorner \nabla_{\nu}d\omega &=  - \sum_{i=1}^n e^i \wedge (\nu \lrcorner \nabla_{\nu}\nabla_{e_i}\omega) + \nabla_{\nu}\nabla_{\nu}\omega - \nu^{\flat} \wedge (\nu \lrcorner \nabla_{\nu}\nabla_{\nu}\omega).
\end{align}
Moreover, using $[e_i, \nu] = \frac{1}{r} e_i$, one calculates
\begin{align} \begin{split} \label{nr}
\nu \lrcorner \nabla_{\nu}\nabla_{e_i}\omega &= \nabla_{e_i}\nabla_{\nu}(\nu \lrcorner \omega) - \frac{1}{r} \nabla_{e_i}(\nu \lrcorner \omega)  + \frac{1}{r^2} e_i \lrcorner \omega - \frac{1}{r} e_i \lrcorner \nabla_{\nu}\omega
 . \end{split}
\end{align}
Inserting this into \eqref{nwoei} and applying the pullback  gives
\begin{align} \begin{split} \label{ngl2}
\iota^* (\nu \lrcorner \nabla_{\nu}d\omega ) &= \iota^*( - d(\nabla_{\nu}(\nu \lrcorner \omega)) + d(\nu \lrcorner \omega) - p \omega + p \nabla_{\nu}\omega + \nabla_{\nu}\nabla_{\nu}\omega) .\end{split} 
\end{align}
Finally, inserting \eqref{ngl1} and \eqref{ngl2} into \eqref{ivmddm}, we obtain
\begin{align*} 
\iota^*\delta^{\mathbb{R}^{n+1}}d\omega 
= \delta^{S^n}d\iota^*\omega + \iota^* \Big( (n-p-1) d( \nu \lrcorner \omega) + p(p-n+1) \omega - n \nabla_{\nu}\omega + d( \nu \lrcorner \nabla_{\nu}\omega) - \nabla_{\nu}\nabla_{\nu}\omega \Big) . 
\end{align*}
Together with \eqref{1tegl}, the proposition follows.
\end{proof}

\begin{kor}
For $\omega \in \Omega^p(\mathbb{R}^{n+1})$, we have
\begin{align*}
\iota^* (\Delta^{\mathbb{R}^{n+1}}\omega) = \Delta^{S^n}(\iota^* \omega) + \iota^*\Big(p(p-n+1)\omega -2 d (\nu \lrcorner \omega) - n \nabla_{\nu}\omega - \nabla_{\nu}\nabla_{\nu}\omega \Big) .
\end{align*}
\end{kor}    

\begin{proof}
This is Proposition \ref{clu} for $\alpha=\beta=1$.
\end{proof}

\begin{deff}
Let $\{e_1, ..., e_{n+1}\}$ be the global standard basis of $T\mathbb{R}^{n+1}$ and $\{e^1, ..., e^{n+1}\}$ the associated dual basis. For each $k \in \mathbb{N}_0$ we define 
\begin{align*}
H^0_k := \{ P \in \mathcal{C}^{\infty}(\mathbb{R}^{n+1}) \mid P \ \text{is a homogeneous polynomial of degree} \ k \ \text{with} \ \Delta_0^{\mathbb{R}^{n+1}}P = 0 \}
\end{align*}
and the space
\begin{align*}
H_k^p := \left\{ \omega = \hspace{-1em} \sum_{1 \leq i_1 < ... < i_p \leq n+1} \hspace{-1em} \omega_{i_1 ... i_p} e^{i_1} \wedge ... \wedge e^{i_p} \in \Omega^p(\mathbb{R}^{n+1}) \ \right| \  \omega_{i_1 ... i_p} \in H^0_k  \ \
\text{and} \ \ \delta^{\mathbb{R}^{n+1}} \omega = 0 \bigg\}
\end{align*}
of all co-closed harmonic homogeneous $p$-forms of degree $k$ on $\mathbb{R}^{n+1}$.
\end{deff}

\begin{Not}
For $f \in \mathcal{C}^{\infty}(\mathbb{R}^{n+1})$ and $x \in \mathbb{R}^{n+1} \setminus \{0\}$, we set $\hat{f}(x) := f(\frac{x}{\|x\|})$.  
Then $\hat{f} \big{|}_{S^n} = f \big{|}_{S^n}$ and $\hat{f}$ is radially constant, i.e.\ $\hat{f}(\lambda x) = \hat{f} (x)$ for all $\lambda > 0$ and $x \in \mathbb{R}^{n+1} \setminus \{0\}$, and therefore $\partial_{\nu}\hat{f} = 0$.
\end{Not}

\begin{deff}
For $k \in \mathbb{N}_0$ we set
\begin{align} \label{lkbp}
\lambda_{\beta,p}^k &:= \beta(k+p)(k+n-p-1) , \\ \label{mkap}
\mu_{\alpha,p}^k &:= \alpha(k+p)(k+n-p+1)
\end{align}
and
\begin{align*}
V_k^p &:=  \iota^* \{ \omega \in H_k^p \mid \nu \lrcorner \omega=0\} , \\
W_k^p &:= \iota^*d(H^{p-1}_{k+1}) .
\end{align*}
\end{deff}

\subsubsection{Eigenforms of $\delta^{S^n} d$}

Let $\omega = \in H_k^p$. Because of $\delta^{\mathbb{R}^{n+1}}\omega = 0$, we have
\begin{align*}
d\delta^{\mathbb{R}^{n+1}}\omega = 0
\end{align*}
and
\begin{align*}
\delta^{\mathbb{R}^{n+1}} d \omega =& (d\delta^{\mathbb{R}^{n+1}} + \delta^{\mathbb{R}^{n+1}} d)\omega - d\underbrace{\delta^{\mathbb{R}^{n+1}}\omega}_{=0} = \Delta_p^{\mathbb{R}^{n+1}} \omega = 0 .
\end{align*}
Moreover, since $\omega$ is homogeneous of degree $k$, we have
\begin{align} \label{nno} 
\nabla_{\nu}\omega =\frac{k}{r} \omega  
\end{align}
and therefore
\begin{align*}
\nabla_{\nu}\nabla_{\nu}\omega =& - \frac{k}{r^2} \omega + \frac{k}{r} \nabla_{\nu} \omega = - \frac{k}{r^2} \omega + \frac{k^2}{r^2} \omega = \frac{k(k-1)}{r^2} \omega .
\end{align*}

For $\omega \in H_0^p$ with $\nu \lrcorner \omega = 0$, we have $\omega = 0$. Hence, $\iota^*\omega = 0$ is not an eigenform of $F_{\alpha \beta,p}^{S^n}$. Therefore, let $\omega \in H_k^p$ with $k \neq 0$ and $\nu \lrcorner \omega = 0$. Proposition \ref{clu} together with the above calculations gives
\begin{align*}
F_{\alpha \beta}^{S^n} \iota^* \omega = \beta (k+p)(k+n-p-1) \iota^* \omega .
\end{align*}

Thus $\iota^*\omega$ is an eigenform of $F_{\alpha \beta,p}^{S^n}$ to the eigenvalue $\lambda_{\beta,p}^k$. Consequently, we have shown that
\begin{align*}
\mathop{\Cup}_{k \in \mathbb{N}} \{ \lambda_{\beta,p}^k \}_{\mathrm{dim}(V_k^p)} \subseteq \mathrm{Spec}(F_{\alpha \beta, p}^{S^n})
\end{align*}
and for all $k \in \mathbb{N}$ that
\begin{align*}
V_k^p = \iota^* \{ \omega \in H_k^p \mid \nu \lrcorner \omega = 0  \} \subseteq \ &\mathrm{Eig}\big(F_{\alpha \beta, p}^{S^n},  \lambda_{\beta,p}^k \big) .
\end{align*}

\subsubsection{Eigenforms of $d\delta^{S^n}$}

Now let $\omega \in d(H_{k+1}^{p-1})$, i.e.\ $\omega = d\eta$ with $\eta \in H^{p-1}_{k+1}$. We have
\begin{align*}
d \delta^{\mathbb{R}^{n+1}} \omega &= d \hspace{-1.4em} \underbrace{\delta^{\mathbb{R}^{n+1}} d}_{= \Delta_p^{\mathbb{R}^{n+1}} - d \delta^{\mathbb{R}^{n+1}}} \hspace{-1.4em} \eta = d \underbrace{\Delta_p^{\mathbb{R}^{n+1}} \eta}_{=0} - \underbrace{d \ d}_{=0} \delta^{\mathbb{R}^{n+1}} \eta = 0 , \\
\delta^{\mathbb{R}^{n+1}} d \omega &= \delta^{\mathbb{R}^{n+1}} \underbrace{d \ d}_{\equiv 0} \eta = 0 
\end{align*}
as well as
\begin{align*}
\iota^*d(\nu \lrcorner \omega) \stackrel{\eqref{ngl1}}{=} d\iota^*(-d(\nu \lrcorner \eta) + (p-1)\eta + \nabla_{\nu}\eta) \stackrel{\eqref{nno}}{=} \iota^*\left((p-1)\omega + \frac{k+1}{r}\omega \right) =  (p+k)\iota^*\omega .
\end{align*}

Since $\omega(e_{i_1}, ..., e_{i_p})$ are homogeneous polynomials of degree $k$ for $1 \leq i_1 < ... i_p \leq n+1$, similar to \eqref{nno} we have
\begin{align*}
\nabla_{\nu}\omega = \frac{k}{r} \omega .
\end{align*}

The last two equations yield
\begin{align*}
\nabla_{\nu} \nabla_{\nu} \omega = \frac{k(k-1)}{r^2} \omega 
\end{align*}

and

\begin{align*}
\iota^*d(\nu \lrcorner \nabla_{\nu} \omega) = d \iota^* \left(\frac{k}{r} \cdot (\nu \lrcorner \omega) \right) = k d\iota^*(\nu \lrcorner \omega) = k(k+p)\iota^*\omega .
\end{align*}

Altogether, with Proposition \ref{clu} we obtain 
\begin{align*}
F_{\alpha \beta}^{S^n} \iota^* \omega = \alpha (k+p)(k+n-p+1) \iota^*\omega .
\end{align*}

So $\iota^*\omega$ is an eigenform of $F_{\alpha \beta, p}^{S^n}$ to the eigenvalue $\mu_{\alpha,p}^k$.
As a result
\begin{align*}
\mathop{\Cup}_{k \in \mathbb{N}_0} \{\mu_{\alpha,p}^k \}_{\mathrm{dim}(W_k^p)} \subseteq \mathrm{Spec}(F_{\alpha \beta, p}^{S^n}) 
\end{align*}
and, for all $k \in \mathbb{N}_0$,
\begin{align*}
W_k^p = \iota^*d(H^{p-1}_{k+1}) \subseteq \mathrm{Eig}\big(F_{\alpha \beta, p}^{S^n}, \mu_{\alpha,p}^k \big) .
\end{align*}

\subsubsection{Proof of the completeness}

Now the question arises wether $\lambda_{\beta, p}^k$ and $\mu_{\alpha, p}^l$ for $k \in \mathbb{N}$ and $l \in \mathbb{N}_0$ already form the entire spectrum of $F_{\alpha \beta, p}^{S^n}$.

\begin{bem} \label{vw}
We already showed that $V_k^p \subseteq \mathrm{Eig}\big(F_{\alpha \beta, p}^{S^n}, \lambda_{\beta, p}^k \big)$ and $W_l^p \subseteq \mathrm{Eig}\big(F_{\alpha \beta, p}^{S^n}, \mu_{\alpha, p}^l \big)$ for all $k \in \mathbb{N}$, $l \in \mathbb{N}_0$, $\alpha, \beta >0$ and $1 \leq p \leq n$, i.e.\ the elements of $V_k^p$ and $W_l^p$ are eigenforms of $F_{\alpha \beta, p}^{S^n}$ for every positive $\alpha$ and $\beta$.
\end{bem}

\begin{lem} \label{ewte}
For all $\alpha, \beta > 0$, $1 \leq p \leq n$ and $k,l \in \mathbb{N}_0$, we have
\begin{enumerate}
\item $\lambda_{\beta, p}^k = \lambda_{\beta, p}^l$ if and only if $k=l$. 
\item $\mu_{\alpha, p}^k = \mu_{\alpha, p}^l$ if and only if $k=l$. 
\item For $\alpha = \beta$, we have $\lambda_{\beta, p}^k = \mu_{\alpha, p}^l$ if and only if $k=l+1$ and $n=2p$. 
\item For $\frac{\alpha}{\beta} \in \mathbb{R} \setminus \mathbb{Q}$, we have $\lambda_{\beta,p}^k \neq \mu_{\alpha,p}^l$.
\end{enumerate}
\end{lem}

\begin{proof} 

It follows directly from the explicit formulas \eqref{lkbp} and \eqref{mkap} that $k \mapsto \lambda^k_{\beta,p}$ and $k \mapsto \mu^k_{\alpha,p}$ are injective functions. This implies i) and ii).

\begin{enumerate}
\setcounter{enumi}{2}
\item ``$\Rightarrow$'': Let $(k+p)(k+n-p-1) = (l+p)(l+n-p+1)$, i.e.\ $\frac{k+p}{l+p} = \frac{l+n-p+1}{k+n-p-1}$. Hence $k \neq l$. If $k \leq l$, then $\frac{k+p}{l+p} \leq 1$ and therefore $\frac{l+n-p+1}{k+n-p-1} \leq 1$. Thus $l \leq k-2$, and due to $l\geq k$ we would obtain $k \leq k-2$, a contradiction. So $k > l$. That is why, similarly to above, we obtain that $l > k-2$, hence $k > l > k-2$. Therefore $l = k-1$. With the assumption it follows that $(k+p)(k+n-p-1) = (k+p-1)(k+n-p)$. Expanding gives that $n=2p$. \\
``$\Leftarrow$'': This direction is trivial. 
\item Let now $\frac{\alpha}{\beta} \in \mathbb{R} \setminus \mathbb{Q}$. If $\lambda_{\beta,p}^k = \mu_{\alpha,p}^l$, then $\frac{\alpha}{\beta}(l+p)(l+n-p+1) = (k+p)(k+n-p-1) \in \mathbb{N}$, which is impossible.\qedhere
\end{enumerate}
\end{proof}

\begin{bem} \label{he}
For $k \in \mathbb{N}$, there is at most one $l \in \mathbb{N}_0$ with $\lambda_{\beta,p}^k = \mu_{\alpha,p}^l$ since, if there was an $l' \in \mathbb{N}_0$ with $\lambda_{\beta,p}^k = \mu_{\alpha,p}^{l'}$, this would imply that $\mu_{\alpha,p}^l = \mu_{\alpha,p}^{l'}$. But by Lemma \ref{ewte}, $l'=l$.
\end{bem}

\begin{lem} \label{og}
Let $k,l \in \mathbb{N}_0$ with $k \neq l$. Then the spaces $V_k^p$, $V_l^p$, $W_k^p$ and $W_l^p$ are pairwise perpendicular with respect to the $L^2$-scalar product. \\
In particular, $V_k^p \oplus V_l^p$, $W_k^p \oplus W_l^p$, $V_k^p \oplus W_l^p$ and $V_k^p \oplus W_k^p$ are direct sums.
\end{lem}

\begin{proof}
Let $k,l \in \mathbb{N}$ with $k \neq l$. Then $V_k^p$ and $V_l^p$ are, due to Remark \ref{vw} and Lemma \ref{ewte}, subspaces of eigenspaces of $F_{\alpha \beta, p}^{S^n}$ for the distinct eigenvalues $\lambda_{\beta,p}^k$ and $\lambda_{\beta,p}^l$ and therefore perpendicular with respect to the $L^2$-scalar product, since $F_{\alpha \beta}$ is self-adjoint. In the same manner, $W_k^p$ and $W_l^p$ are subspaces of eigenspaces for the distinct eigenvalues $\mu_{\alpha, p}^k$ and $\mu_{\alpha,p}^l$ and hence perpendicular. \\
Now let $\alpha, \beta > 0$ such that that $\frac{\alpha}{\beta} \in \mathbb{R} \setminus \mathbb{Q}$. Then, due to Lemma \ref{ewte}, $\lambda_{\beta,p}^k \neq \mu_{\alpha,p}^l$ and $\lambda_{\beta,p}^k \neq \mu_{\alpha,p}^k$, so the statement follows for $V_k^p$ and $W_l^p$, respectively $V_k^p$ and $W_k^p$.
\end{proof}

Now we can show that the eigenvalues of the operator $F_{\alpha \beta,p}^{S^n}$ we derived in the previous both sections already form its whole spectrum.

\begin{thm} \label{thms}
Let $1 \leq p \leq n$ and $\alpha, \beta >0$. The spectrum of the operator $F_{\alpha \beta,p}^{S^n}$ is given by
\begin{align*}
\mathrm{Spec}(F_{\alpha \beta,p}^{S^n}) &= \left( \mathop{\Cup}_{k \in \mathbb{N}} \{ \lambda_{\beta,p}^k \}_{\mathrm{dim}(V_k^p)} \right) \Cup \left( \mathop{\Cup}_{k \in \mathbb{N}_0} \{ \mu_{\alpha,p}^k \}_{\mathrm{dim}(W_k^p)} \right) 
\end{align*}
and the corresponding eigenspaces are, for $k \in \mathbb{N}$,
\begin{align*}
\mathrm{Eig}\big(F_{\alpha \beta,p}^{S^n}, \lambda_{\beta,p}^k \big) = 
\begin{cases}
V_k^p, &\text{if} \ \lambda_{\beta,p}^k \neq \mu_{\alpha,p}^l \ \text{for all} \ l \in \mathbb{N}_0 \\
V_k^p \oplus W_l^p, &\text{if} \ \lambda_{\beta,p}^k = \mu_{\alpha,p}^l \ \text{for one} \ l \in \mathbb{N}_0
\end{cases}
\end{align*}
and, for $k \in \mathbb{N}_0$,
\begin{align*}
\mathrm{Eig}\big(F_{\alpha \beta,p}^{S^n}, \mu_{\alpha,p}^k \big) &= 
\begin{cases}
W_k^p, &\text{if} \ \mu_{\alpha,p}^k \neq \lambda_{\beta,p}^l \ \text{for all} \ l \in \mathbb{N} \\
W_k^p \oplus V_l^p, &\text{if} \ \mu_{\alpha,p}^k = \lambda_{\beta,p}^l \ \text{for one} \ l \in \mathbb{N}
\end{cases}.
\end{align*}
\end{thm}

\begin{proof} 
By \cite[Lemma 3.1]{bc} we have $H_k^p = \{ \omega \in H_k^p \mid \nu \lrcorner \omega = 0\} \oplus d(H^{p-1}_{k+1})$ for all $k \in \mathbb{N}_0$. In \cite[Korollar 6.6]{it} it is shown that $\iota^* (\bigoplus_{k \in \mathbb{N}_0} H_k^p )$ is dense in $\Omega^p(S^n)$ and, therefore, dense in $\Omega^p_{L^2}(S^n)$. \\
Let $\lambda \in \mathrm{Spec}(F_{\alpha \beta,p}^{S^n})$ and $\omega \in \mathrm{Eig}(F_{\alpha \beta,p}^{S^n}, \lambda)$ with $\omega \neq 0$. Due to the density and Lemma \ref{og}, we have
\begin{align*}
\mathrm{Eig}(F_{\alpha \beta,p}^{S^n}, \lambda) \subseteq \Omega^p_{L^2}(S^n) = \overline{\iota^*\Big(\bigoplus_{k \in \mathbb{N}_0}H_k^p \Big)}^{L^2} = \overline{\bigoplus_{l \in \mathbb{N}} V_l^p \oplus \bigoplus_{l \in \mathbb{N}_0} W_l^p}^{L^2},
\end{align*}
 i.e.\ $\omega = \sum_{l \in \mathbb{N}}v_l + \sum_{l \in \mathbb{N}_0}w_l$ for $v_l \in V_l^p$ and $w_l \in W_l^p$. Therefore, with Remark \ref{vw} it follows that
\begin{align*}
\lambda \left( \sum_{l \in \mathbb{N}}v_l + \sum_{l \in \mathbb{N}_0}w_l \right) = \lambda \omega = F_{\alpha \beta}^{S^n} \omega = \sum_{l \in \mathbb{N}} \lambda_{\beta,p}^l v_l + \sum_{l \in \mathbb{N}_0} \mu_{\alpha,p}^l w_l .
\end{align*}
As a result, $\lambda_{\beta,p}^l = \lambda$ for all $l \in \mathbb{N}$ with $v_l \neq 0$ and $\mu_{\alpha,p}^l = \lambda$ for all $l \in \mathbb{N}_0$ with $w_l \neq 0$ and, therefore,
\begin{align*}
\mathrm{Eig}(F_{\alpha \beta,p}^{S^n}, \lambda) \subseteq \overline{\bigoplus_{\substack{l \in \mathbb{N}: \\ \lambda_{\beta,p}^l = \lambda}} V_l^p \oplus \bigoplus_{\substack{l \in \mathbb{N}_0: \\ \mu_{\alpha,p}^l = \lambda}} W_l^p}^{L^2} = \bigoplus_{\substack{l \in \mathbb{N}: \\ \lambda_{\beta,p}^l = \lambda}} V_l^p \oplus \bigoplus_{\substack{l \in \mathbb{N}_0: \\ \mu_{\alpha,p}^l = \lambda}} W_l^p .
\end{align*}
Since $\omega \neq 0$ we have $\lambda = \lambda_{\beta,p}^l$, for some $l \in \mathbb{N}$, or $\lambda = \mu_{\alpha,p}^l$, for some $l \in \mathbb{N}_0$, i.e.\ $\lambda$ is one of the already known eigenvalues. Hence, with $\lambda_{\beta,p}^l$ for $l \in \mathbb{N}$ and $\mu_{\alpha,p}^l$ for $l \in \mathbb{N}_0$, we have already found the entire spectrum of $F_{\alpha \beta,p}^{S^n}$. 

\medskip
In order to determine the associated eigenspaces, we choose $k \in \mathbb{N}$ fixed and consider $\lambda := \lambda_{\beta,p}^k$. Lemma \ref{ewte} then tells us that $\lambda_{\beta,p}^l = \lambda_{\beta,p}^k$ for an $l \in \mathbb{N}$ if and only if $l=k$. Furthermore, due to Remark \ref{he}, $\mu_{\alpha,p}^l = \lambda_{\beta,p}^k$ can be true for at most one $l \in \mathbb{N}_0$.  \\
We have shown that
\begin{align*}
\mathrm{Eig}(F_{\alpha \beta,p}^{S^n}, \lambda_{\beta,p}^k) \subseteq 
\begin{cases}
V_k^p, &\text{if} \ \lambda_{\beta,p}^k \neq \mu_{\alpha,p}^l \ \text{for all} \ l \in \mathbb{N}_0 \\
V_k^p \oplus W_l^p, &\text{if} \ \lambda_{\beta,p}^k = \mu_{\alpha,p}^l \ \text{for one} \ l \in \mathbb{N}_0
\end{cases} .
\end{align*}
We have $V_k^p \subseteq \mathrm{Eig}\big(F_{\alpha \beta,p}^{S^n}, \lambda_{\beta,p}^k \big)$ and $W_l^p \subseteq \mathrm{Eig}\big(F_{\alpha \beta,p}^{S^n}, \mu_{\alpha,p}^l \big) = \mathrm{Eig}\big(F_{\alpha \beta,p}^{S^n}, \lambda_{\beta,p}^k\big)$, if $\lambda_{\beta,p}^k = \mu_{\alpha,p}^l$ for some $l \in \mathbb{N}_0$. Hence, in this case $V_k^p \oplus W_l^p \subseteq \mathrm{Eig}\big(F_{\alpha \beta,p}^{S^n}, \lambda_{\beta,p}^k \big)$. This yields
\begin{align*}
\mathrm{Eig}(F_{\alpha \beta,p}^{S^n}, \lambda_{\beta,p}^k) = 
\begin{cases}
V_k^p, &\text{if} \ \lambda_{\beta,p}^k \neq \mu_{\alpha,p}^l \ \text{for all} \ l \in \mathbb{N}_0 \\
V_k^p \oplus W_l^p, &\text{if} \ \lambda_{\beta,p}^k = \mu_{\alpha,p}^l \ \text{for one} \ l \in \mathbb{N}_0
\end{cases} .
\end{align*}
Analogously one shows that
\begin{align*}
\mathrm{Eig}(F_{\alpha \beta,p}^{S^n}, \mu_{\alpha,p}^k) &= 
\begin{cases}
W_k^p, &\text{if} \ \mu_{\alpha,p}^k \neq \lambda_{\beta,p}^l \ \text{for all} \ l \in \mathbb{N} \\
W_k^p \oplus V_l^p, &\text{if} \ \mu_{\alpha,p}^k = \lambda_{\beta,p}^l \ \text{for one} \ l \in \mathbb{N}
\end{cases} . \qedhere
\end{align*}
\end{proof}

\begin{bem} \label{speig} $ $
\begin{enumerate}
\item For $\alpha= \beta$ and $n=2p$ we have seen in Lemma \ref{ewte} that $\lambda_{\beta}^{k+1} = \mu_{\alpha}^k$ for all $k \in \mathbb{N}_0$. Thus, in this case we have that, for all $k \in \mathbb{N}_0$,
\begin{align*}
\mathrm{Eig}\big(F_{\alpha \alpha,p}^{S^n}, \alpha (k+p)(k+p+1)\big) = V_{k+1}^p \oplus W_k^p 
\end{align*}
and
\begin{align*}
\mathrm{Spec}(F_{\alpha \alpha,p}^{S^n}) = \mathop{\Cup}_{l \in \mathbb{N}_0} \{\alpha (l+p)(l+p+1)\}_{\mathrm{dim}(V_{l+1}^p \oplus W_l^p)} .
\end{align*}
\item If $\alpha = \beta$ and $n \neq 2p$ or if $\frac{\alpha}{\beta} \in \mathbb{R}\setminus \mathbb{Q}$, we have, again by Lemma \ref{ewte}, that for all $k \in \mathbb{N}$ and $l \in \mathbb{N}_0$
\begin{align*}
\mathrm{Eig}(F_{\alpha \beta,p}^{S^n}, \lambda_{\beta,p}^k) = V_k^p \ \ \ \ \ \text{and} \ \ \ \ \
\mathrm{Eig}(F_{\alpha \beta,p}^{S^n}, \mu_{\alpha,p}^l) = W_l^p .
\end{align*}
\item The eigenvalues of $F_{\alpha \beta,p}^{S^n}$ are all positive since $n-p \geq 0$.
\end{enumerate}

\end{bem}

\subsection{Multiplicities}

With the following proposition we can directly read off the geometric multiplicities of the eigenvalues.

\begin{satz} \label{mults}
For all $1 \leq p \leq n$ and $k \in \mathbb{N}$, we have
\begin{align*}
\mathrm{dim}(V_k^p) = \frac{(n+k-1)!(n+2k-1)}{p!(k-1)!(n-p-1)!(n+k-p-1)(k+p)}
\end{align*}
and, for all $k \in \mathbb{N}_0$, 
\begin{align*}
\mathrm{dim}(W_k^p) = \frac{(n+k)!(n+2k+1)}{(p-1)!k!(n-p)!(n+k-p+1)(k+p)} .
\end{align*}
\end{satz}

\begin{proof}
For $n \neq 2p$ and $1 \leq p \leq n$, we have by Remark \ref{speig}  that
\begin{align*}
\mathrm{dim}(V_k^p) = \mathrm{dim}\big(\mathrm{Eig}(F_{11,p}^{S^n}, \lambda_{1,p}^k)\big)
\end{align*}
for all $k \in \mathbb{N}$, and, for all $k \in \mathbb{N}_0$, that
\begin{align*}
\mathrm{dim}(W_k^p) = \mathrm{dim}\big(\mathrm{Eig}(F_{11,p}^{S^n}, \mu_{1,p}^k)\big) .
\end{align*}
In \cite[table I]{bc} the multiplicities of the eigenvalues of $\Delta^{S^n}_p = F_{11,p}^{S^n}$ are listed.
\end{proof}


\subsection{Isospectrality}

\subsubsection{$F_{\alpha \beta}$ on spheres of different radii}

We consider $n$-dimensional spheres $S^n_r$ of radius $r > 0$ embedded in $\mathbb{R}^{n+1}$ via the canonical inclusions $\iota_r: S^n_r \rightarrow \mathbb{R}^{n+1}$. First of all, the following proposition says that the spectrum of $F_{\alpha \beta,p}^{S^n_r}$ emanates from the one of $F_{\alpha \beta,p}^{S^n}$ by multiplication with the factor $\frac{1}{r^2}$. 

\begin{satz} \label{sskmr}
Let $\alpha, \beta, r > 0$ and $1 \leq p \leq n$. Then
\begin{align*}
\mathrm{Spec}(F_{\alpha \beta, p}^{S^n_r}) = \frac{1}{r^2} \cdot \mathrm{Spec}(F_{\alpha \beta, p}^{S^n}) =   \mathrm{Spec} \! \left(F_{\frac{\alpha}{r^2} \frac{\beta}{r^2}, p}^{S^n} \right).
\end{align*}
\end{satz}

We omit the proof as it is an easy calculation.

\begin{deff}
For $\alpha, \beta, r > 0$, $1 \leq p \leq n$ and $k \in \mathbb{N}_0$ we set
\begin{align*}
\lambda_{\beta,p, r}^k &:= \frac{\beta}{r^2}(k+p)(k+n-p-1) \ \ \  \ \text{and} \\
\mu_{\alpha,p, r}^k &:= \frac{\alpha}{r^2}(k+p)(k+n-p+1) .
\end{align*}
\end{deff}

Below we show that the operators $F_{\alpha \beta}$ on spheres of different radii can never be isospectral.

\begin{satz} \label{abgleich}
Let $\alpha, \beta, r, r' > 0$ and $1 \leq p \leq n$.
Then $F_{\alpha \beta,p}^{S^n_r}$ and $F_{\alpha \beta,p}^{S^n_{r'}}$ are isospectral if and only if $r = r'$.
\end{satz}

\begin{proof}
``$\Rightarrow$'': Since $F_{\alpha \beta,p}^{S^n_r}$ and $F_{\alpha \beta,p}^{S^n_{r'}}$ have the same spectrum, in particular their smallest eigenvalues coincide. 
In the case that $\frac{\alpha}{\beta} \geq \frac{(p+1)(n-p)}{p(n-p+1)}$, we have that $\lambda_{\beta,p, r}^1 \leq \mu_{\alpha,p, r}^0$ and $\lambda_{\beta,p, r'}^1 \leq \mu_{\alpha,p, r'}^0$. It follows that $\min(\mathrm{Spec}(F_{\alpha \beta,p}^{S^n_r})) = \min(\mathrm{Spec}(F_{\alpha \beta,p}^{S^n_{r'}}))$ if and only if $\frac{\beta(p+1)(n-p)}{r^2} = \lambda_{\beta,p, r}^1 = \lambda_{\beta,p, r'}^1 = \frac{\beta(p+1)(n-p)}{r'^2}$, i.e.\ if and only if $r=r'$.
If $\frac{\alpha}{\beta} < \frac{(p+1)(n-p)}{p(n-p+1)}$, we have $\lambda_{\beta,p, r}^1 > \mu_{\alpha,p, r}^0$ and $\lambda_{\beta,p, r'}^1 > \mu_{\alpha,p, r'}^0$. Hence $\min(\mathrm{Spec}(F_{\alpha \beta,p}^{S^n_r})) = \min(\mathrm{Spec}(F_{\alpha \beta,p}^{S^n_{r'}}))$ if and only if $\frac{\alpha p (n-p+1)}{r^2} = \mu_{\alpha,p,r}^0 = \mu_{\alpha,p, r'}^0 = \frac{\alpha p (n-p+1)}{r'^2}$, i.e.\ if and only if $r=r'$. \\
``$\Leftarrow$'': This direction is trivial.
\end{proof}

\subsubsection{Variation of parameters}

We now discuss the question how the spectra of two operators $F_{\alpha \beta,p}^{S^n_r}$ and $F_{\alpha' \beta',p'}^{S^n_{r'}}$ are related for different parameters $\alpha, \alpha', \beta, \beta', r, r' > 0$ and $1 \leq p, p' \leq n$. As a first step, we will fix the radii of the spheres.

\begin{bem}
Proposition \ref{symn=2} implies that for all $\alpha, \beta, r > 0$ and $n=2p$
\begin{align*}
\mathrm{Spec}(F_{\alpha \beta,p}^{S^n_r}) = \mathrm{Spec}(F_{\beta \alpha, p}^{S^n_r}) .
\end{align*}
\end{bem}

\begin{thm} \label{varpar1}
Let $\alpha, \alpha', \beta, \beta', r > 0$ and $1 \leq p \leq n$ with $n \neq 2p$. Then $F_{\alpha \beta,p}^{S^n_r}$ and $F_{\alpha' \beta',p}^{S^n_r}$ are isospectral if and only if $(\alpha, \beta) = (\alpha', \beta')$. \\
For $n=2p$ the statement holds with $\{\alpha, \beta\} = \{\alpha', \beta'\}$ instead of $(\alpha, \beta) = (\alpha', \beta')$.
\end{thm}

\begin{proof}
``$\Rightarrow$'': Let at first $n \neq 2p$. We show that the map
\begin{align*}
f: (0,\infty) \times (0,\infty) \rightarrow \mathcal{M} : (\gamma, \delta) \mapsto \mathrm{Spec}(F_{\gamma \delta,p}^{S^n_r})
\end{align*}
is injective. To this end, let $M$ be in the image of $f$. Due to Remark \ref{speig}, we have $m := \min(M) > 0$. Furthermore, due to Theorem \ref{thms} and Proposition \ref{sskmr}, the multiplicity $\mathrm{mult}(m)$ of $m$ satisfies
\begin{align*}
\mathrm{mult}(m) \in \{ \mathrm{dim}(V_1^p), \mathrm{dim}(W_0^p), \mathrm{dim}(V_1^p) + \mathrm{dim}(W_0^p)\} \stackrel{\ref{mults}}{=} \left\{ \binom{n+1}{p+1}, \binom{n+1}{p}, \binom{n+2}{p+1} \right\}.
\end{align*}
Since $1 \leq p \leq n$ and $n \neq 2p$, the set in fact consists of three different elements.
\begin{enumerate}
\item[$\bullet$] In the case that $\mathrm{mult}(m) = \mathrm{dim}(V_1^p)$, we set $\delta := \frac{r^2m}{(p+1)(n-p)}$. Let $M' := M \setminus \Cup_{k \in \mathbb{N}} \{\lambda_{\delta,p,r}^k\}_{\mathrm{dim}(V_k^p)}$ and $m' := \min(M')$. Then we define $\gamma := \frac{r^2m'}{p(n-p+1)}$. 
\item[$\bullet$] If $\mathrm{mult}(m) = \mathrm{dim}(W_0^p)$, we put $\gamma := \frac{r^2m}{p(n-p+1)}$ and $\delta := \frac{r^2m'}{(p+1)(n-p)}$. Here $M' := M \setminus \Cup_{k \in \mathbb{N}_0} \{\mu_{\gamma,p,r}^k\}_{\mathrm{dim}(W_k^p)}$ and $m' := \min(M')$. 
\item[$\bullet$] In the last case that $\mathrm{mult}(m) = \mathrm{dim}(V_1^p) + \mathrm{dim}(W_0^p)$ let $\gamma := \frac{r^2m}{p(n-p+1)}$ and $\delta := \frac{r^2m}{(p+1)(n-p)}$. 
\end{enumerate}
In each case, $M = \mathrm{Spec}(F_{\gamma \delta, p}^{S^n_r})$. Here $\gamma$ and $\delta$ are unique by construction. Thus $f$ is injective. Since by assumption $f(\alpha, \beta) = f(\alpha', \beta')$, it follows that $(\alpha, \beta) = (\alpha', \beta')$. 

\medskip
Now let $n=2p$. We consider the map
\begin{align*}
\widetilde{f}: \big\{C \subset (0,\infty) \mid \# C \in \{1,2\} \big\} \rightarrow \mathcal{M} : \{\gamma, \delta\} \mapsto \mathrm{Spec}(F_{\gamma \delta,p}^{S^{2p}_r}) .
\end{align*}
Let $M$ be in its image. Then $m:= \min(M) > 0$. We set $\gamma := \frac{r^2m}{p(p+1)}$. Let $M' := M \setminus \Cup_{k \in \mathbb{N}_0} \{\mu_{\gamma,p,r}^k\}_{\mathrm{dim}(W_k^p)}$ and $m' := \min{M'}$. We define $\delta := \frac{r^2m'}{p(p+1)}$. Then $M = \mathrm{Spec}(F_{\gamma \delta,p}^{S^2_r})$, whereat $\{\gamma, \delta\}$ is unique by construction. Consequently, $\widetilde{f}$ is injective. Hence, $\widetilde{f}(\{\alpha, \beta\}) = \widetilde{f}(\{\alpha', \beta'\})$ implies that $\{\alpha, \beta\} = \{\alpha', \beta'\}$.

``$\Leftarrow$'': The opposite directions are trivial.
\end{proof}

For different radii, the spectra are related as follows:

\begin{kor}
Let $\alpha, \alpha', \beta, \beta', r, r' > 0$ and $1 \leq p \leq n$ such that $F_{\alpha \beta,p}^{S^n_r}$ and $F_{\alpha' \beta',p}^{S^n_{r'}}$ are isospectral and $n \neq 2p$. Then $r=r'$ if and only if $(\alpha, \beta) = (\alpha', \beta')$. \\
For $n=2p$ the statement holds with $\{\alpha, \beta\} = \{\alpha', \beta'\}$ instead of $(\alpha, \beta) = (\alpha', \beta')$. 
\end{kor}

\begin{proof}
The first direction is just Theorem \ref{varpar1} and the opposite direction Proposition \ref{abgleich}.
\end{proof}

\begin{satz}
Let $\alpha, \alpha', \beta, \beta', r, c > 0$, $1 \leq p \leq n$ and $n \neq 2p$. Then $F_{\alpha \beta,p}^{S^n_r}$ and $F_{\alpha' \beta',p}^{S^n_{cr}}$ are isospectral if and only if $(\alpha', \beta') = (c^2\alpha, c^2\beta)$. \\
For $n=2p$ the statement holds with $\{\alpha', \beta'\} = \{c^2 \alpha, c^2 \beta\}$ instead of $(\alpha', \beta') = (c^2 \alpha, c^2 \beta)$.
\end{satz}

\begin{proof}
Proposition \ref{sskmr} tells us that $F_{\alpha \beta,p}^{S^n_r}$ and $F_{\alpha' \beta',p}^{S^n_{cr}}$ are isospectral if and only if
\begin{align*}
\mathrm{Spec}(F_{\alpha \beta,p}^{S^n}) &= r^2 \mathrm{Spec}(F_{\alpha \beta,p}^{S^n_r}) = r^2 \mathrm{Spec}(F_{\alpha' \beta',p}^{S^n_{cr}}) = \frac{1}{c^2} \mathrm{Spec}(F_{\alpha' \beta',p}^{S^n}) = \mathrm{Spec} \!\left( \! F_{\frac{\alpha'}{c^2} \frac{\beta'}{c^2},p}^{S^n} \! \right) .
\end{align*}
Due to  Theorem \ref{varpar1}, for $n \neq 2p$ this is equivalent to
$(\alpha, \beta) = \left(\frac{\alpha'}{c^2}, \frac{\beta'}{c^2}\right)$,
and for $n=2p$ to
$\{\alpha, \beta\} = \left\{\frac{\alpha'}{c^2}, \frac{\beta'}{c^2} \right\}$.
\end{proof}


\bibliographystyle{alpha}
\bibliography{Literature}

\end{document}